\documentclass{amsart}
\usepackage{psfrag}
\usepackage{color}
\usepackage{amsmath}
\usepackage{graphicx,amssymb}
\newcommand{\bfi}{\bfseries\itshape}
\usepackage{wrapfig}

\usepackage{hyperref}
\hypersetup{
  pdftitle={Variational Fluids},
  pdfauthor={Pavlov et al.},
  pdfsubject={Variational Integrators for Fluid Simulation},
  pdfkeywords={},
  pdfpagemode=UseThumbs,
  baseurl={http://www.geometry.caltech.edu/},
  colorlinks=true,
  linkcolor=black,
  anchorcolor=black,
  citecolor=blue,
  filecolor=black,
  menucolor=black,
  urlcolor=blue,
  bookmarksopen=false,
}

\def\f{\varphi}
\def\ftilde{\bar{\varphi}}

\def\a{\alpha}
\def\g{\gamma}
\def\e{\varepsilon}
\def\half{\frac12}

\def\lmat{(}
\def\rmat{)}
\def\linner{(\! (}
\def\rinner{)\!)}
\def\la{\langle}
\def\ra{\rangle}
\def\llangle{\langle\!\langle}
\def\rrangle{\rangle\!\rangle}
\def\lla{\langle\!\langle}
\def\rra{\rangle\!\rangle}
\def\o{\omega}
\def\O{\Omega}
\def\A{\mathcal A}
\def\R{\mathbb R}
\def\M{\mathcal M}
\def\N{\mathcal N}
\def\N1{\mathcal N}
\def\S{\mathcal S}

\def\L{\mathcal L}

\def\TO{\mathcal{D}} 
\def\dTO{\mathfrak{D}} 
\def\SDiff{{\rm SDiff}}
\def\SVect{{\rm SVect}}

\def\frakSO{\mathfrak{so}}
\def\one{\mathbf 1}
\newcommand{\F}[1]{{^{\scriptscriptstyle #1\!\!}F}} 
\newcommand{\Ff}[1]{{^{\scriptscriptstyle #1\!\!}f}}
\newcommand{\dd}{\mathrm{\mathbf d}} 
\newcommand{\ii}{\mathrm{\mathbf i}} 
\newcommand{\Tr}{\mathrm{Tr}}

\newcommand{\Id}{\mathrm{Id}}
\renewcommand{\div}{\mathop{\mathrm{div}}}
\newcommand{\longto}[1]{\xrightarrow[#1]{}} 
\newcommand{\tendsto}[1]{\underset{#1}{\rightsquigarrow}} 
\newcommand{\diag}{\mathop{\mathrm{diag}}}
\newcommand{\mesh}{\mathbb{M}}

\newtheorem{Def}{Definition}
\newtheorem{theorem}{Theorem}
\newtheorem{lemma}{Lemma}

\title{Structure-Preserving Discretization of Incompressible Fluids}
\author[Pavlov, Mullen, Tong, Kanso, Marsden, and Desbrun]{D. Pavlov$^\blacktriangle$, P. Mullen$^\blacktriangle$, Y. Tong$^\blacklozenge$,\\ E. Kanso$^\blacktriangledown$, J.E. Marsden$^\blacktriangle$, M. Desbrun$^\blacktriangle$\\
{\tiny $^\blacktriangle$Caltech \hspace*{.1in} $^\blacklozenge$MSU \hspace*{.1in} $^\blacktriangledown$USC}}

\begin{document}

\begin{abstract}
The geometric nature of Euler fluids has been clearly identified and extensively studied over the years, culminating with Lagrangian and Hamiltonian descriptions of fluid dynamics where the configuration space is defined as the volume-preserving diffeomorphisms, and Kelvin's circulation theorem is viewed as a consequence of Noether's theorem associated with the particle relabeling symmetry of fluid mechanics. However computational approaches to fluid mechanics have been largely derived from a numerical-analytic point of view, and are rarely designed with structure preservation in mind, and often suffer from spurious numerical artifacts such as energy and circulation drift. In contrast, this paper geometrically derives discrete equations of motion for fluid dynamics from first principles in a purely Eulerian form. Our approach approximates the group of volume-preserving diffeomorphisms using a finite dimensional Lie group, and associated discrete Euler equations are derived from a variational principle with non-holonomic constraints. The resulting discrete equations of motion yield a structure-preserving time integrator with good long-term energy behavior and for which an exact discrete Kelvin's circulation theorem holds.
\end{abstract}

\maketitle


\section{Introduction}
The geometric nature of Euler fluids has been extensively studied in the literature in works of Arnold, Ebin-Marsden and others; however the geometric-differential standpoint of these studies sharply contrasts with the numerical approaches traditionally used in Computational Fluid Dynamics (CFD). In particular, methods based on particles, vortex particles, staggered Eulerian grids, spectral elements, as well as hybrid Lagrangian-Eulerian formulations were not designed with structure preservation in mind --- in fact, recent work pinpoints the loss of Lagrangian structures as a major numerical impediment of current CFD techniques~\cite{HoLe2009}. In contrast, structure preserving methods (so-called geometric integrators) have recently become popular in the context of Lagrangian dynamics in solid mechanics. Based on discrete versions of Hamilton's principle and its variants, they have been shown to capture the dynamics of the mechanical system they discretize without traditional numerical artifacts such as loss of energy or momenta.

While the variational principles for incompressible fluid mechanics are best expressed in a Lagrangian formalism, computational efficiency often calls for an \emph{Eulerian treatment} of fluid computations to avoid numerical issues inherent to deforming meshes. In order to circumvent these issues without giving up structure preservation, a new Eulerian formulation of discrete fluid mechanics is thus needed.

Guided by the variational integrators used in the Lagrangian setting, this paper introduces a discrete, structure-preserving theory for incompressible perfect fluids based on Hamilton-d'Alembert's principle. Such a discrete variational approach to fluid dynamics guarantees invariance under the particle-relabeling group action and gives rise to a discrete form of Kelvin's circulation theorem. Due to their variational character, the resulting numerical schemes also exhibit good long-term energy behavior. In addition, the resulting schemes are not difficult to implement in practice (see Figure~\ref{fluid_teaser}), and we will derive particular instances of numerical update rules and provide numerical results. We will favor formalism over smoothness in the exposition of our approach in order to better elucidate the correspondences between continuous and discrete expressions.

\begin{figure}[h]
\centering
\includegraphics[width=\textwidth]{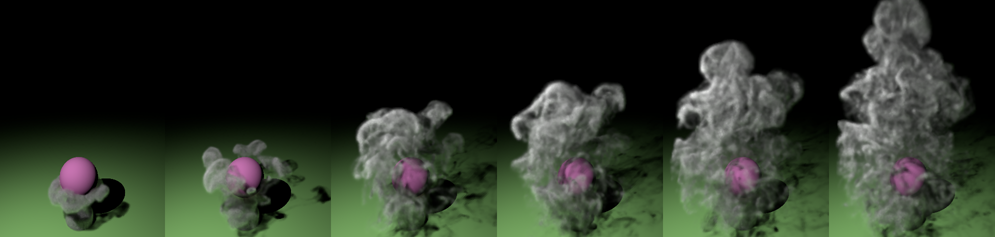}
\caption{\footnotesize Our geometric approach to discretizing the dynamics of incompressible fluids leads
to discrete, structure-preserving, Lie group integrators. Here, six frames of an animation simulating heated smoke rising around a round obstacle in a closed box of incompressible fluid.\vspace*{-2mm}}
\label{fluid_teaser}
\end{figure}

\subsection{Brief Review of the Continuous Case.} %
Let $M\subset \R^n$ be an arbitrary compact manifold, possibly with boundary (where $n$ denotes the dimension of the domain, typically, 2 or 3), and $\SDiff(M)$ be the group of smooth volume-preserving diffeomorphisms on $M$. As was shown in \cite{Arnold1966},  the motion of an ideal incompressible fluid in $M$ may be described by a geodesic curve $g_t$ in $\SDiff(M)$. That is, $\SDiff(M)$ serves as the configuration space---a particle located at a point $x_0 \in M$ at time $t=0$ travels to $g_t(x_0)$ at time $t$. Being geodesics, the equations of motion naturally derive from Hamilton's stationary action principle:
\begin{equation}
\label{fluidLagrangian}
\delta \int_{0}^{1} L(g, \dot{g}) \; dt  = 0\quad\mbox{where} \;\;\;
L(g, \dot{g})   =  \frac{1}{2} \int_M \|\dot{g}\|^2 \; dV
\end{equation}
subject to arbitrary variations $\delta g$ vanishing at the endpoints. Here, the Lagrangian $L(g,\dot g)$ is the kinetic energy of the fluid and $dV$ is the standard volume element on $M$.
As this Lagrangian is invariant under {\em particle relabeling}---that is, the action of $\SDiff(M)$ on itself by composition on the {\it right}, the principle stated in Eq.~\eqref{fluidLagrangian} can be rewritten in
{\em reduced (Eulerian) form} in terms of the Eulerian velocity  $v = \dot{g} \circ g^{-1}$:
\begin{equation}
\delta \int_{0}^{1} l(v) \; dt  = 0 \;\;\;\; \mbox{where} \;\;\;
l(v)   =  \frac{1}{2} \int_{M_0}  \|v\| ^2 \; dV
\end{equation}
subject to constrained variations $\delta v = \dot\xi + [v,\xi]$ (called \emph{Lin constraints}), where $\xi$ is an arbitrary divergence-free vector field---an element of the Lie algebra of the group of volume preserving diffeomorphisms---and $[ \, , ]$ is the Jacobi Lie bracket (or vector field commutator). There is a complex history behind this reduced variational principle which was first shown for general Lie groups by \cite{MaSc1993}; see also~\cite{Arnold1966, EbMa1970, ArKh1992, MaRa1999}). As stated above, the reduced Eulerian principle is more attractive in computations because it involves a fixed Eulerian domain (mesh); however, the constrained variations necessary in this context complicates the design of a variational Eulerian algorithm.

\subsection{Overview and Contributions.}
While time integrators for fluid mechanics are often derived by approximating equations of motion, we instead follow the geometric principles described above and discretize the configuration space of incompressible fluids in order to derive the equations of motion through the principle of stationary action. Our approach uses an Eulerian, finite dimensional representation of volume-preserving diffeomorphisms that encodes the displacement of a fluid from its initial configuration using special orthogonal, signed stochastic matrices. From this particular discretization of the configuration space, which forms a finite dimensional Lie group, one can derive a right-invariant discrete equivalent to the Eulerian velocity through its Lie algebra, i.e., through antisymmetric matrices whose columns sum to zero. After imposing non-holonomic constraints on the velocity field to allow transfer only between neighboring cells during each time update, we apply the Lagrange-d'Alembert principle (a variant of Hamilton's principle applicable to non-holonomic systems) to obtain the discrete equations of motion for our fluid representation. As we will demonstrate, the resulting Eulerian variational Lie-group integrator is structure-preserving, and as such, has numerous numerical properties, from momentum preservation (through a discrete Noether theorem) to good long-term energy behavior.

\begin{figure}[h!]
\centering
\includegraphics[width=2.5in]{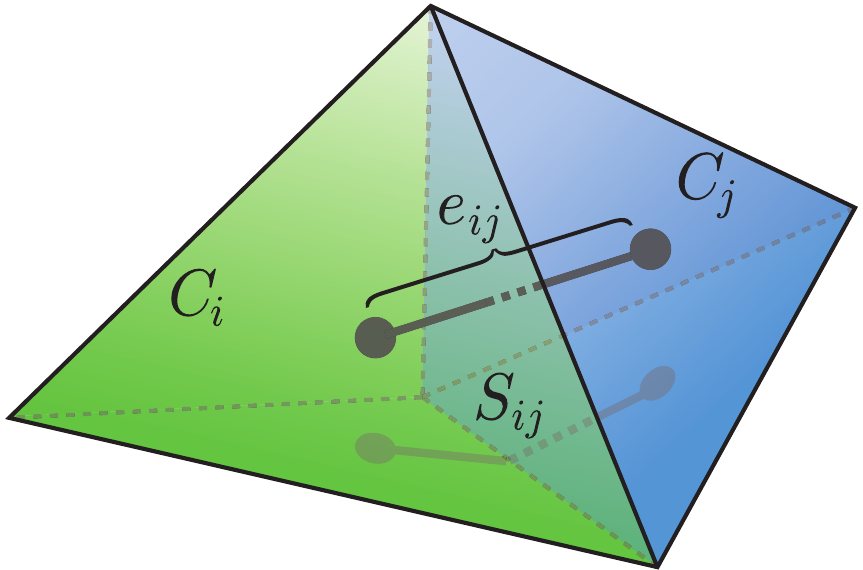}
\caption{\footnotesize Spatial Discretization: two cells $C_i$ and $C_j$, with their common face $S_{ij}= C_i \cap C_j$ of area $|S_{ij}|$ and its dual edge $e_{ij}$ of length $|e_{ij}|$.} \label{adjacentcells}
\end{figure}

\subsection{Notations.}
\label{sec:Notations}
The spatial discretization (mesh), either simplicial (tetrahedra) or regular (cubes), will be denoted $\mesh$, with $N$ being the number of $n$-dimensional cells $\{C_i\}_{i=1, \ldots , N}$ in $\mesh$. The {\bfi size} of a mesh will refer to the maximum diameter $h$ of its cells. The Lebesgue measure will be denoted by $|.|$. Thus, $|C_i|$ is the volume of cell $C_i$, $|C_i \cap C_j|$ is the area of the face common to $C_i$ and $C_j$, etc (see Figure~\ref{adjacentcells}). The {\bfi dual} of $\mesh$ is the circumentric dual cell complex~\cite{Hirani2003}, formed by connecting the circumcenters of each cell $C_i$ based on the connectivity of $\mesh$. We will further assume that the mesh $\mesh$ is Delaunay with well shaped elements~\cite{ToWoAlDe2009} to avoid degeneracies of its orthogonal dual as well as to simplify the exposition. We will also use the term {\bfi regular grid} (or Cartesian grid) to designate a mesh that consists of cells that are $n$-dimensional cubes of equal size. The notation $N(i)$ will denote the set of indices of cells neighboring cell $C_i$, that is, cell $C_i$ shares a face with cell $C_j$ iff $j \in N(i)$.  We will say that a pair of cells $C_i, C_j$ is {\bfi positively oriented around an edge $e$} if they share a face containing $e$ and they are oriented such that they ``turn'' clockwise around the edge when viewed along the oriented edge.  The same term will be used similarly for triplets of cells $C_i, C_j, C_k$ where $i,k \in N(j)$ and all three cells contain edge $e$.

\smallskip

The notation $(.,.)$ and $\langle .,. \rangle$ will respectively refer to the $L^2$ inner product of vectors and the pairing of one-forms and vector fields, while their discrete counterparts will be denoted by $\linner .,. \rinner$ and $\lla .,. \rra$. Table~\ref{tab:variables} summarizes the main variables used in the remainder of this paper, along with their meaning and representation.

\begin{table}[h!]
{\footnotesize
\begin{center}
\begin{tabular}{|c||p{4.3cm}|p{5.7cm}|}
\hline
\textbf{Symbol}  &  \textbf{Meaning}     & \textbf{Representation} \\
\hline \hline
$M$ & Domain of motion & $M\subset\R^n$\\ \hline $n$ & Dimension of
the domain & $n \in \mathbb{N}$ \\ \hline
$\SDiff(M)$ & Configuration space of ideal fluid & Volume-preserving diffeomorphisms on $M$\\ \hline
$\SVect(M)$ & Tangent space of $\SDiff(M)$ at Id & Divergence-free vector fields on $M$\\ \hline \hline
$\mesh$ & Mesh discretizing domain $M$ & Simplicial or regular mesh\\ \hline
$N$ & Number of cells in $\mesh$ & $N\in\mathbb{N}$\\ \hline $C_i$ & Cell $\#i$ of $\mesh$&  Tetrahedron or cube in $3D$\\ \hline
$\O$ & Discrete analog of volume form & Diagonal matrix of cell volumes, $\O_{ii}\!=\!|C_i|$\\ \hline
$\TO(\mesh)$ & Discrete configuration space & $\O$-orthogonal signed stochastic matrices\\ \hline
$\dTO(\mesh)$ & Lie algebra of $\TO(\mesh)$ & $\O$-antisymmetric null-row matrices\\ \hline
$q$ & Discrete configuration & Matrix $\in \TO(\mesh) \subset \mathop{GL}(N) \subset \mathcal{M}^N$ \\ \hline
$A$  & Discrete Eulerian velocity $-\dot{q} q^{-1}$ &  Matrix $\in \dTO(\mesh) \subset \mathfrak{gl}(N) = \mathcal{M}^N$\\ \hline
$\F{k}$ & Discrete $k$-form &  $N$-dimensional tensor of order $(k+1)$\\ \hline
$\N1$ & Space of matrices with sparsity based on cell adjacency & Constrained set of matrices, with $\A_{ij}\ne 0\Rightarrow j \in N(i)$\\  \hline
$\S$ & Space of sparse discrete velocities& Constrained set of velocities, $\S=\dTO(\mesh)\cap\N1$\\ \hline
\end{tabular}\end{center}}
\caption{\footnotesize Physical/Geometric meaning of the basic (continuous and discrete) variables used throughout this document.\vspace*{-6mm}}
\label{tab:variables}
\end{table}

\medskip
\noindent
{\bf Acknowledgments.} We thank Daryl Holm and Yann Br{\'e}nier for helpful early discussions and input, Evan Gawlik for generating the energy plots, and Keenan Crane for generating our 2D tests. This research was partially supported by NSF grants CMMI-0757106, CCF-0811373, and DMS-0453145.

\section{Discrete Volume Preserving Diffeomorphisms}
\label{sec:VPDiff}

We first introduce a finite dimensional approximation to the infinite dimensional
Lie group of volume preserving diffeomorphisms that tracks the amount of
fluid transfered from one cell to another while preserving two key properties: volume
and mass preservation.

\subsection{Finite Dimensional Configuration Space.}
Suppose that the domain $M$ is approximated by a mesh
$\mesh$. Our first step in constructing a discrete representation of ideal fluids is
to approximate $\SDiff(M)$ with a finite dimensional Lie group in such
a way that the elements of the corresponding Lie algebra can be considered as a
discretization of divergence-free vector fields. To achieve this goal, we will
\emph{not discretize the diffeomorphism $g$ itself}, but rather the associated operator
$U_g: L _{2} \rightarrow L _{2}$ defined by $\varphi (x) \mapsto \varphi (g ^{-1} (x) )$. Here $L^2=L^2(M,\R)$ is the space of square integrable real valued functions on $M$. An important property of $U_g$ is given by the following lemma, which follows from the change of variables formula.
\begin{lemma}[Koopman's lemma\footnote{Many dynamical properties of $g$, such as ergodicity, mixing etc., can be studied using spectral properties of $U_g$. The idea of using methods of Hibert spaces to study dynamical systems was fist suggested by Koopman \cite{Koopman1931} and is usually called Koopmanism; it is closely related to the Perron-Frobenius methodology.}] If the diffeomorphism $g$ is volume-preserving, then $U_g$ is a unitary operator on $L^2$.
\end{lemma}
Another important property of $U_g$ is that it preserves constants, i.e.,
$U_g C=C$ for every constant function $C$, which can be seen as mass preservation for fluids. Next we present an approach to discretize
this operator $U_g$ while respecting its two defining properties.
\medskip

\noindent
{\bf Discrete Functions.} To discretize the operator $U_g$ we first need to
discretize the space on which $U_g$ acts. Since the mesh $\mesh_h$ splits the
domain of motion $M$ into $N$ cells $C_i$ of maximum diameter $h$, a function
$\f \in C^0(M;\R)$ can be approximated by a step function $\ftilde$, constant within
each cell of the mesh, through a map $R_{\mesh_h}: C^0(M;\R) \rightarrow {\rm step \; functions}$, which averages $\f$ per cell:
\begin{equation*}
R_{\mesh_h}:\f \mapsto \ftilde,\quad \ftilde = \sum_i \left[\frac{1}{\Omega_i} \int_{C_i}\f \right] \chi_{C_i},
\end{equation*}
where $\chi_{C_i}$ is the indicator function for the cell $C_i$, and $\Omega_i=|C_i|$ is the volume of cell $C_i$. Since the space of all step functions on $\mesh_h$ is isomorphic to $\R^{N}$, we can consider the step functions as vectors: using the map $P_{\mesh_h}: L _{2} \rightarrow \mathbb{R}^{N}$ defined by
\begin{equation}\label{ProjectorP}
\lmat P_{\mesh_h}\f\rmat_i=\frac{1}{\Omega_i}\int_{C_i}\f,
\end{equation}
we can define a vector $\f_h = P_{\mesh_h}\f$ of size $N$ to represent the step function $\ftilde$. To reconstruct a step function from an arbitrary vector $\f_h\in\R^N$ we
define an operator $S_{\mesh_h}: \mathbb{R}^{N} \rightarrow L _{2}$ by
\begin{equation*}
\left( S_{\mesh_h}\f_{h} \right) (x) = \lmat \f_h\rmat_i,\quad\text{if }x\in C_i.
\end{equation*}
Thus, the operators $R_{\mesh_h}$, $P_{\mesh_h}$ and $S_{\mesh_h}$ are
related through:
\begin{equation*}
R_{\mesh_h}=S_{\mesh_h}P_{\mesh_h}.
\end{equation*}
The vector $\f_h$ will be called a {\bfi discrete function} as it provides
an approximation of a continuous function $\f$: when $h \to 0$,
\begin{equation*}
\|S_{\mesh_h}\f_h-\f\|_{C^0}\to 0.
\end{equation*}

We also introduce a discrete approximation of the continuous $L^2$ inner product of
functions $\la \f,\psi\ra=\int_M\f\psi$ through:
\begin{equation}\label{InnerProduct}
\la\f_h,\psi_h\ra=\sum_i\O_i \; \lmat\f_h\rmat_i \lmat\psi_h\rmat_i.
\end{equation}

\noindent
{\bf Discrete Diffeomorphisms.}
Using the fact that a matrix $q_h\in \M^N$ (here $\M^N$ is the space of real valued
$N\times N$ matrices) acts on a vector $\f_h$, we will say that $q_h$
approximates $U_g$ if $S_{\mesh_h}(q_h\f_h)$ is close to $U_g\f$:
\begin{Def} Consider a family of meshes $\mesh_h$ of size $h$, each consisting of $N_h$
cells $C_i^h$. We will say that a family of matrices $q_h\in\M^N$ approximates
a diffeomorphism $g\in\SDiff(M)$ (and denote this property as: $q_h
\rightsquigarrow g$) if the following is true:
\begin{equation*}
S_{\mesh_h}(q_hP_{\mesh_h}\f)\stackrel{C^0}{\longrightarrow} U_g\,\f
\quad\text{for every }\f\in C(M;\R).
\end{equation*}
\label{def:discreteDiffeo}
\end{Def}

In order to better respect the continuous structures at play, we further enforce
that our discrete configuration space of diffeomorphisms satisfies two key
properties of $U_g$: volume-preservation, reflecting the fact that $U_g$ is
unitary, and total mass preservation, as $U_g$ preserves constants. We will
thus only consider matrices $q$ that:

\begin{itemize}
\item \emph{preserve the discrete $L^2$ inner product of functions}, i.e.,
$$\langle q\f_h,q\psi_h\rangle=\langle \f_h,\psi_h\rangle,$$
where the inner product of discrete functions is defined by
Eq.~\eqref{InnerProduct}. Denoting
$$\Omega=\begin{pmatrix}
|C_1|&0&\ldots&0\\
0&|C_2|&\ldots&0\\
\vdots&\vdots&\ddots
&\vdots\\ 0&0&\ldots&|C_N|\\
\end{pmatrix},$$
note that this discrete notion of volume preservation
directly implies that for our mesh $\mesh_h$ a volume preserving matrix $q$ satisfies
$$ q^T\Omega q=\Omega. $$
The matrix $q$ is thus \emph{$\O$-orthogonal}, restricted to matrices of
determinant $1$.

\item \emph{preserve constant vectors} (i.e., vectors having all coordinates equal) as well:
$$q\one=\one, \quad \text{where: } \one=\left(\begin{smallmatrix}1\\ \vdots\\ \\ 1\end{smallmatrix}\right).$$
The matrix $q$ must thus be \emph{signed stochastic} as well.
 \end{itemize}

Consequently, the finite dimensional space of matrices we will use to discretize
volume-preserving diffeomorphisms has the following definition:
\begin{Def}\label{appr}
Let $\mesh$ be a mesh consisting of cells $C_i$, $i=1,\ldots,N$ and $\O$ be the diagonal matrix consisting of volumes of the cells, i.e., $\O_{ii}=|C_i|$ and $\O_{ij}=0$ when $i\ne j$ (we will abusively use the shorter notation $\Omega_i$ to denote a diagonal element of $\O$ for simplicity in what follows). We will call a matrix $q\in\M^N$ {\bfi volume-preserving} and {\bfi constant-preserving} with respect to the mesh $\mesh$ if, for all $i$ in $\{1,\ldots,N\},$
\begin{equation}\label{stochastic2}
q^T\Omega q=\Omega.
\end{equation}
and
\begin{equation}\label{stochastic1}
\sum_j q_{ij}=1,
\end{equation}

The set of all such $\Omega$-orthogonal, signed stochastic matrices of determinant $1$ will be
denoted $\TO(\mesh)$, and will be used as a discretization of the configuration
space $\SDiff(M)$.
\end{Def}

Our finite dimensional configuration space $\TO(\mesh)$ for fluid dynamics is thus the intersection of two Lie groups: the $\O$-orthogonal group, and the group of invertible stochastic matrices; therefore, it is a \emph{Lie group}. Note that if all cells of $\mesh$ have the same volume, i.e., $\Omega=\Omega_0 \; \Id$, then a matrix $q\in\TO(\mesh)$ is \emph{orthogonal} in the usual sense and the equality~\eqref{stochastic2} implies $\sum_i q_{ij}=1$. For such meshes (which include Cartesian grids), the matrix $q$ is {\it signed doubly-stochastic}.

\medskip

\noindent
{\bf Remark.}\label{DS} An alternate, arguably more intuitive way to discretize
a diffeomorphism $g\in\SDiff(M)$ on a mesh $\mesh$ would be to define a matrix
$q$ as:
\begin{equation*}
q_{ij}(g)\equiv\frac{|g^{-1}(C_j)\cap C_i|}{|C_i|}.
\end{equation*}
This discretization also satisfies by definition a discrete preservation of mass and
a (different) notion of volume preservation. While it has the added benefit of
enforcing that $q$ has no negative terms (therefore respecting the positivity of
$U_g$), the class of matrices it generates is, unfortunately, only a semi-group,
which would be an impediment for establishing a variational treatment of fluids as an inverse map will be needed in the Eulerian formulation. So instead, we take the \emph{orthogonal} part of this matrix as our configuration
(which can be obtained in practice through the \emph{polar decomposition}). Notice that
polar factorization has often been proposed in the context of fluids (see, e.g.,~\cite{Brenier1991}),
albeit for more general non-linear Hodge-like decomposition.

\subsection{Discrete Velocity Field.}
Now that we have established a finite dimensional configuration space
$\TO(\mesh)$, we describe its associated Lie algebra, and show that
elements of this Lie algebra provide a discretization of divergence-free vector
fields $\SVect(M)$. We will assume continuous time for simplicity, but a fully
discrete treatment of space and time will be introduced in Section~\ref{sec:discretetime}.

Consider a smooth path in the space of volume-preserving diffeomorphisms
$g_t\in\SDiff(M)$ with $g_0=\Id$, and let $q_h(t)$ be an approximation of $g_t$,
i.e., for any piecewise constant function $\f_h^0$ approximating a smooth
function $\f^0\in C^1(M,\R),$ a discrete version of $\f^0\circ g^{-1}_t=\f(t)$ is given by
\[
\f_h(t)=q_h(t)\,\f_h^0.
\]
Assuming $q_h(t)$ is smooth in time, we define its {\bfi Eulerian velocity} $A_h(t)$ to
be
$$A_h(t)=-\dot q_h(t)\,q_h^{-1}(t),$$
thus yielding
$$\dot \f_h(t)=-A_h(t)\,\f_h(t).$$

Since $\frac{d}{dt}(\f^0\circ g^{-1}_t)=-\la\dd \f(t),v_t\ra = - \mathbf{L}_{v_t} \f$, where
$v_t= \dot{g} _{t} \circ g^{-1}_t$ and $\mathbf{L}_{v_t}$ is the Lie derivative, the matrix $A_h(t)$
represents an approximation of the Eulerian velocity field $v_t$, which
motivates the following definition:

\begin{Def}\label{approxlv}
Consider a one-parameter family of volume-preserving diffeomorphisms
$g_t\in\SDiff(M)$ and the associated time-dependent vector field
$v_t = \dot{g} _{t} \circ g^{-1}_t \in \SVect(M)$. Consider a family of meshes $\mesh_h$ of size $h$
consisting of cells $C_i^h$ and an operator $P_{\mesh_h}:C(M;\R)\to
\R^{N_h}$ defined by Eq.~\eqref{ProjectorP}.

We will say that a family of matrices $A_h(t)\in\M^{N_h}$ approximates a vector
field $v_t$ (denoted by $A_h(t)\rightsquigarrow v_t$) if the following statement is
true:
\begin{equation*}
S_{\mesh_h}(A_h(t)P_{\mesh_h}\f)\stackrel{C^0}{\longrightarrow} \mathbf{L}_{v_t}\f\quad\text{for every }\f\in C^\infty(M;\R).
\end{equation*}
\label{def:AApproximation}
\end{Def}
\vspace{-0.15in}

\noindent
{\bf Remark.} The choice of the minus sign in the definition of $A_h(t)$ stems from the fact that $q_h(t)$ represents $U_g$ (thus, $g^{-1}$ in essence). Since $\mathbf{L}_v = -\dot{U_g} U_g^{-1}$, we picked the sign to make $A_h(t)$ represents $\mathbf{L}_v$, consistent with the continuous case.
\medskip

If a curve of matrices $q(t)$ belongs to the configuration space $\TO(\mesh)$ (i.e., if $q(t)$
is $\O$-orthogonal signed stochastic), then its associated $A$ belongs to its Lie
algebra that we denote as $\dTO(\mesh)$. Matrices from this
Lie algebra inherit the properties that \emph{their rows must sum to zero}:
$$\sum_j A_{ij}=0\quad\text{(preservation of mass)},$$
and \emph{they are $\O$-antisymmetric}:
$$A^T\O+\O A=0\quad\text{(preservation of volume)}.$$
These two properties can be intuitively understood as discrete statements that
$A$ represents an advection, and the vector field representing this advection is
divergence-free. Lie algebra elements for arbitrary simplicial meshes will be
called \emph{null-row $\O$-antisymmetric   matrices}. Note that if the mesh
is regular ($\Omega=\Omega_0 \; \Id$), $q$ belongs to the orthogonal group and
the matrix $A$ has to be \emph{antisymmetric} with both its rows and columns
summing to zero (``doubly null'').

The link between convergence of $A_h(t)$ to $\mathbf{L}_{v_t}$ and
convergence of $q_h(t)$ to $U_{g_t}$ is described by the following lemma.
\begin{lemma}
Consider the setup of Definition~\textup{\ref{def:AApproximation}} and suppose a family of matrices $A_h(t)\in\dTO(\mesh_h)$ approximates the Lie derivative $\mathbf{L}_{v_t}$ (in the sense of Definition~\ref{approxlv}) uniformly in $t$ when $t\in [0,T]$ for some $T>0$.

Then there is a family of matrices $q_h(t)\in\TO_{\mesh_h}$ such that
$A_h(t)=-\dot q_h(t)\,{q_h(t)}^{-1}$ and $q_h(t)$ approximates $g_t$ (in the sense
of Definition~\ref{def:discreteDiffeo}).
\end{lemma}

\begin{proof}
Consider a family of smooth functions $\f(t,x)$ satisfying the advection
equation
\begin{equation*}
\dot \f(t,x)=-\mathbf{L}_{v_t}\f(t,x).
\end{equation*}
Suppose that $\ftilde(0,x)=S_{\mesh_h}P_{\mesh_h}\f(0,x)$ is an approximation to
$\f(0,x)$ with
\begin{equation*}
\sup_{x\in M}|\ftilde(0,x)-\f(0,x)|<\epsilon_1
\end{equation*}
and that $\f_h(t)=P_{\mesh_h}\ftilde(t,x)$ satisfies the discrete advection
equation
\begin{equation*}
\dot{\f}_h(t)=-A_h(t)\,\f_h(t).
\end{equation*}

Since $A_h(t)$ approximates $\mathbf{L}_{v_t}$, given $\e_2>0$, we can choose $h$ such that
\begin{equation*}
\|S_{\mesh_h}(A_h(t)\f_h(t))-\mathbf{L}_{v_t}\f\|<\epsilon_2,\quad\text{for all }t\in [0,T].
\end{equation*}
Therefore,
\begin{equation*}
\|S_{\mesh_h}(\dot{\f}_h(t))(x)-\dot\f(t,x)\|<\epsilon_2,\quad\text{for all }t\in [0,T]
\end{equation*}
and
\begin{equation*}
\|S_{\mesh_h}(\f_h(t))(x)-\f(t,x)\|<\epsilon_2+\epsilon_2t.
\end{equation*}
Thus, we have shown that $\f_h(t,x)$ approximates $\f(t,x)$. However, $\f(t,x)$ satisfies
\begin{equation*}
\f(t,x)=U_{g_t}\f(0,x),
\end{equation*}
and $\f_h$ satisfies
\begin{equation*}
\f_h(t)=q(t)\,\f_h(0),
\end{equation*}
where $q(t)$ is the matrix satisfying the equation
\begin{equation*}
\dot q(t)=-A_h(t) q(t).
\end{equation*}
Therefore, we see that $q(t)\f(0)$ approximates $U_{g_t}\f(0,x)$. Thus,
$A_h(t) \rightsquigarrow v_t$ implies that $q(t) \rightsquigarrow g_t$.
\end{proof}

\subsection{Discrete Commutator.} A space-discrete flow that approximates a continuous flow $g(t)\in\SDiff(M)$ is defined to be a smooth path $q(t)\in \TO(\mesh)$ in the
space of $\O$-orthogonal signed stochastic matrices, such that
$q(t)\rightsquigarrow g(t)\in\SDiff(M)$ (see Definition~\ref{appr}) and
$A(t)=-\dot q(t)\,q^{-1}(t)\rightsquigarrow v_t=\dot g_t(g^{-1}_t)$ (see
Definition~\ref{approxlv}). It is straightforward to show that the
Lie algebra structure of the space of divergence-free vector fields is
preserved by our discretization. Indeed, if two matrices $A$ and $B$
approximate vector fields $u$ and $v$ then their commutator $[A,B]$
approximates the commutator of the Lie derivative operators:
$$[A,B]\to \mathbf{L}_u\mathbf{L}_v-\mathbf{L}_v\mathbf{L}_u.$$
Since $\mathbf{L}_u\mathbf{L}_v-\mathbf{L}_v\mathbf{L}_u=\mathbf{L}_{[u,v]}$, we obtain $[A,B]\rightsquigarrow [u,v],$
where $[.,.]$ denotes both the commutator of vector fields and the commutator of matrices.
This property will be very useful to deal with Lin constraints later on.

\subsection{Non-holonomic Constraints (NHC)}\label{sec:NHC}
For a smooth path $q(t),$ the matrix $A(t)$ describes the infinitesimal
exchanges of fluid particles between any pair of cells $C_i$ and $C_j$. We will thus
 assume that $A_{ij}$ is non-zero \emph{only} if cells $C_i$ and $C_j$
share a common boundary, i.e., are immediate neighbors. This sparsity will be numerically advantageous later on to reduce the computational complexity of the resulting integration schemes. We thus choose to \emph{restrict discrete paths $\{q(t)\}$ on $\TO(\mesh)$ to those for which $A(t)$ satisfies this constraint}\footnote{Although we will adopt this sparsest form of the velocity in this paper, there may be advantages in considering larger non-zero neighborhoods in future work.}. In other words, we \emph{only} consider null-row $\O$-antisymmetric matrices satisfying the constraints as valid discrete vector fields. The non-zero elements $A_{ij}$ of these matrices correspond to boundaries between adjacent cells $C_i$ and $C_j$, and can be interpreted as \emph{directional transfer densities (per second) from $C_i$ to $C_j$}---they could abusively be called ``fluxes'' on regular grids; but we will make the proper link with the integrals of the velocity field over mesh faces in the next section.

More formally, we define the {\bfi constrained set} $S_q\subset T_q\TO(\mesh)$ as
the set of matrices corresponding to exchanges between neighboring cells only,
i.e., $\dot q\in S_q$ if and only if $\lmat \dot q q^{-1}\rmat_{ij}\ne 0$ implies that
the cells $C_i$ and $C_j$ are neighbors. In this case the matrix $A$ is defined by a
set of non-zero values $A_{ij}$ defined on faces between adjacent cells $C_i$ and
$C_j$. As mentioned previously, to indicate their adjacency, we will write that $j\!\in\!N(i)$
and $i\!\in\!N(j)$, where $N(k)$ refers to the set of indices of adjacent cells to cell $C_k$ in the
mesh $\mesh$. We will say that a matrix $A$ belongs to the class $\N1$ if
$A_{ij}\ne 0$ implies $j\!\in\! N(i)$. Finally, we will denote by $\S\equiv
S_{\text{Id}}=\dTO(\mesh)\cap\N1$, the constrained set at the identity. Consequently, our
treatment of fluid dynamics will only consider matrices $A$ in $\S \subset
\dTO(\mesh)$, i.e.,  matrices in $\dTO(\mesh)$ satisfying the sparsity constraints.

Note that if two matrices $A$ and $B$ both satisfy the constraints, their
commutator need \emph{not}: while the element of the
commutator corresponding to any pair of cells which are more than two cells
away is zero, the element $[A,B]_{ij}$ may be non-zero when cells $C_i$ and
$C_j$ are ``two cells away'' from each other since
\[[A,B]_{ij}=\sum_k(A_{ik}B_{kj}-B_{ik}A_{kj}).\]
Notice that the commutator is zero for neighboring cells since
$A_{kk}=B_{kk}=0$ due to their $\O$-antisymmetry. Writing
$[\S,\S]=\{[A,B]\mid A,\,B\in\S\}$, one sees that $\S\cap [\S,\S]=\{\mathbf 0\}$, where
$\mathbf 0$ is the zero matrix. Therefore, the constraints we just defined are
\emph{non-holonomic}.

\medskip

\noindent
{\bf Remark.}
When a discrete vector field $A$ is in $\S$, the non-zero values $\O_i
A_{ij}$ of the antisymmetric matrix $\O A$ can be understood as dual
$1$-chains, i.e., $1$-dimensional chains on the dual of
$\mesh$~\cite{Munkres1984}. This connection with 1-chains will become crucial
later when dealing with advection of curves to derive a discrete Kelvin's theorem
in Section~\ref{sec:Kelvin}.

\subsection{Relation Between Elements of $A$ and Fluxes.}\label{sec:AasFlux}
Suppose we have a family of discrete flows $q_h(t)$ which approximates a flow $g_t\in\SDiff(M)$ such that $A_h(t)=- \dot q_h(t)\,{q_h(t)}^{-1}$ approximates $\mathbf{L}_{v_t}$ and satisfies the NHC. Let's see how individual elements $\lmat A_h\rmat_{ij}(t)$ of $A_h(t)$ are related to spatial values of $v_t$. Recall that
\[\dot\f_h(t)=-A_h(t)\,\f_h(t)\]
is a discrete version of the advection equation
\[\dot\f=-\mathbf{L}_{v_t}\f\]
and $A_h(t)\f_h(t)\to \mathbf{L}_{v_t}\f$ in the $C^0$ norm. But it also means that
$\lmat\Omega A_h(t)\f_h(t)\rmat_i$ is an approximation to the integral
$\int_{C_i}\mathbf{L}_{v_t}\f_t$, i.e.,
\begin{equation}
\sum_{j\in N(i)} \O_i \lmat A_h\rmat_{ij}(t)\f_j(t) \approx \int_{C_i}\mathbf{L}_{v_t}\f_t
\overset{\nabla \cdot v_t = 0}{=} \int_{\partial C_i}\f_t \;(v_t,\vec{n})
\label{eq:intFlux}
\end{equation}
where $\vec{n}$ is the normal vector to the boundary of $C_i$ and $(.,.)$ denotes the inner product of vectors. However,
\begin{equation*}
\int_{\partial C_i}\f_t \;(v_t,\vec{n}) \approx \sum_{j\in N(i)}\frac{1}{2}(\f_i+\f_j) \int_{S_{ij}}(v_t,\vec{n}_{ij}) \overset{\nabla \cdot v_t = 0}{=} \sum_{j\in N(i)}\frac{1}{2}\f_j \int_{S_{ij}}(v_t,\vec{n}_{ij}).
\end{equation*}
where $S_{ij}$ is the face shared by cells $C_i$ and $C_j$, and $\vec{n}_{ij}$ the normal vector to $S_{ij}$ oriented from $C_i$ to $C_j$. By comparing this result to equation \eqref{eq:intFlux}, it is clear that an element $\O_i\lmat A_h\rmat_{ij}(t)$ can be considered (up to a constant) as an approximation to the \emph{flux} of a vector field $v(t)$ through $S_{ij}$:
\[\O_iA_{ij}(t)\approx \frac{1}{2}\int_{S_{ij}} (v_t,\vec{n}_{ij}).\]

We know that $\int_{S_{ij}} (v_t,\vec{n})\approx (v_t(x_{ij}),\vec{n}_{ij})S_{ij} +O(h^2)$,
where $x_{ij}$ is the barycenter of the boundary $S_{ij}$ and $|S_{ij}|$ is the area
of $S_{ij}$. Therefore, we obtain that, up to a constant dependent on local mesh measures,
$\lmat A_h\rmat_{ij}$ approximates the flux through the boundary between $C_i$ and $C_j$, i.e.,
\begin{equation*}
\lmat A_h\rmat_{ij}(t)\approx (v_t(x_{ij}),\vec{n}_{ij})\frac{|S_{ij}|}{2 \O_i}.
\end{equation*}

In the case of a Cartesian grid of size $h$ this formula simplifies to:
\begin{equation*}
\lmat A_h\rmat_{ij}(t)\approx \frac{(v_t(x_{ij}),\vec{n}_{ij})}{2h}.
\end{equation*}

\subsection{Towards Lagrangian Dynamics with Non-holonomic Constraints.}
\label{sec:hint}%
One of the goals of this paper is to approximate geodesic flows on $\SDiff(M)$ by Lagrangian flows on $\TO(\mesh)$. To achieve this goal, we first need to define a Lagrangian $\L_h(q,\dot q)$ such that
\begin{equation}\label{LapproxL}
\L_h(q,\dot q)\to\int_M\half \|v\|^2 dV\quad\text{when }-\dot q q^{-1}\rightsquigarrow v
\end{equation}
and
\begin{equation}\label{dLapproxdL}
\delta \L_h(q,\dot q)\to\delta\int_M \half \|v\|^2 dV\quad\text{when }-\dot q
q^{-1}\rightsquigarrow v \text{ and } \delta(-\dot q q^{-1})\rightsquigarrow \delta v.
\end{equation}
Such a Lagrangian, depending only on $A=-\dot q q^{-1}$ to mimic the continuous case, can then be used to formulate fluid dynamics through a discrete \emph{Lagrange-d'Alembert principle} (to account for the non-holonomic constraint we impose on the sparsity of our Eulerian velocity approximation):
\begin{equation*}
\delta \int_0^1 \L_h(q,\dot q)dt=0 \text{ with } \left\{
\begin{aligned}
\delta & q\in S_q\\
\delta & q(0)=\delta q(1)=0.
\end{aligned}
\right.
\end{equation*}
Note that the constraint on the variations of $q$ will induce a constraint on the variations of $A$, giving rise to a discrete version of the well-known Lin constraints of the form $\delta A=\dot B+[A,B],$ with  $B=-\delta q \; q^{-1}$ (see Section~\ref{DiscreteLinConstraints}).

However, we will show in later sections that coming up with a proper Lagrangian will require great care. As is typical with nonholonomic systems, the dynamics on $\TO(\M)$ will depend strongly on the values of $\partial \L_h / \partial A$ (i.e., the matrix with $\partial \L_h / \partial A_{ij}$ as its $(i,j)$ element) \emph{outside} of the constraint set $\S$ because of the commutator present in the Lin constraints. In particular, a conventional discretization of the kinetic energy via the sum of all the squared fluxes on the grid would lead to a matrix $ \partial \L_h / \partial A$ with only values on pairs of adjacent cells, resulting in no dynamics. Instead, the Lagrangian \emph{must} depend on values $A_{ij}$ where $i\notin N(j)$.

To satisfy properties~(\ref{LapproxL}) and~(\ref{dLapproxdL}), we will look for a Lagrangian $\L_h$ of the form
\[
\L_h(A)=\frac12 \linner A,A\rinner,
\]
where the discrete $L^2$-inner-product $\linner \cdot,\cdot\rinner$ will be defined to satisfy the following properties (where $(\cdot,\!\cdot)$ denotes the continuous inner product of vector fields): for all $A, B \in \S $,
\[
\linner A,B \rinner=\linner B,A \rinner
\to \int_M (u,v) dV,\quad\text{when } A\rightsquigarrow u,\,B\rightsquigarrow v\]
and for all $A, B,C \;\in  \S$
\begin{equation}
\linner A,[B,C]\rinner \to\int_M \!(u, [v,w]) dV \!=
\! \int_M \!-\dd u^\flat(v,w) dV, \,
\quad \text{when } \left\{ \!\!\begin{array}{l}
  A\rightsquigarrow u \\
  B\rightsquigarrow v\\
  C\rightsquigarrow w
  \end{array} \right.
\label{eq:pairingWithCommutator}
\end{equation}
where $\flat$ is the continuous flat operator (see for instance~\cite{AbMaRa1988}). These properties will guarantee that conditions~(\ref{LapproxL}) and~(\ref{dLapproxdL}) are satisfied, and will lead to the proper dynamics. In the next section we will present a discretization of differential forms and a few operators acting on them to help us construct the discrete $L^2$-inner product (or equivalently, the discrete flat operator $\flat$).

\section{Structure-Preserving Spatial Field Discretization}
We now introduce a discrete calculus consistent with our discretization of vector
fields. Unlike previous discrete exterior calculus approaches, mostly based on
chains and cochains (see \cite{DeKaTo2005,BoHy2005,ArFaWi2006} and references therein), we clearly
distinguish between discrete vector fields and discrete forms acting on them.
Moreover, our notion of forms will need to act not only on vector fields satisfying
the NHC (being thus very reminiscent of the chain/cochain approach), but also
on vector fields resulting from a commutator as imposed by the Lin constraints. We
also introduce a discrete contraction operator $\ii_v$ and a discrete Lie derivative $\mathbf{L}_v$
to complete our set of spatial operators---we will later show that the algebraic
definition of our Lie derivative matches its dynamic counterpart as expected.
We will not make any distinction in symbols between the discrete and continuous exterior calculus
operators ($\ii_v$, $\mathbf{L}_v$, $\dd$, $\flat$, etc) as the context will make their meaning clear.

\subsection{Discrete Zero-forms.}
\label{sec:0forms}
In our context, a discrete $0$-form is a function $\F{0}$ that is piecewise
constant per cell as previously defined in Section~\ref{sec:VPDiff}. Note that its
representation is a vector of $N$ cell values,
$$\F{0} = ( \F{0}_1, \F{0}_2, \hdots, \F{0}_N)^T,$$
where $\F{0}_i$ represents the value of the function $\F{0}$ in cell $C_i$. Also, the
volume integral of such a discrete $0$-form is obtained by weighting the value of each
cell by the Lebesgue measure of this cell, and summing all contributions:
\begin{equation*}
\quad \int_M \F{0} \; dV = \; \sum_{i=1}^{N} \Omega_i \F{0}_i.
\end{equation*}

\medskip

\noindent
{\bf Remark.}
Our definition of $0$-forms is no different from dual $0$-cochains in dimension $n$ as
used extensively in, e.g.,~\cite{Munkres1984,DeKaTo2005}. They naturally pair with
dual $0$-chains (i.e., linear combinations of cell circumcenters).

\subsection{Discrete One-forms.}
As the space $\TO(\mesh)$ of matrices is used to discretize vector fields,
a natural way to discretize one-forms is to also use matrices to respect the
duality between these two entities. Moreover, it is in line with the previous
definition for $0$-forms that were encoded as a 1-tensor: $1$-forms will now be
encoded by a $2$-tensor. Notice that this is also reminiscent of the approximation
$TM\approx M\times M$ used in discrete mechanics~\cite{MaWe2001}.
\medskip

\noindent
\textbf{Discrete Contraction.}
We define the contraction operator by a discrete vector field $A$, acting on a
discrete one-form $\F{1}$ to return a discrete zero-form, as:
\begin{equation}
\ii_A \F{1} \equiv \diag (A \F{1}^T)\stackrel{\mathrm{def}}{=}((A \F{1}^T)_{11},\ldots,(A \F{1}^T)_{NN})^T.
\label{eq:contractionSimple}
\end{equation}
Notice the metric-independence of this definition, and that if the discrete vector
field contains only non-zero terms for neighboring cells, any term $\lmat
\F{1}\rmat _{ij}$ where cell $C_i$ and cell $C_j$ are not neighbors does not
contribute to the contraction. In this case, the value of the resulting $0$-form for
cell $C_i$ is thus: $\lmat\ii_A \F{1}\rmat_i = \sum_{j\in N(i)} A_{ij} \F{1}_{ij}$,
which is a local sum of the natural pairings of $\F{1}$ and $A$ on each face of
cell $C_i$.

\medskip

\noindent
\textbf{Discrete Total Pairing.}
With this contraction defined, we derive a total pairing between a discrete $1$-form and a
discrete vector field as:
\begin{equation*}
\lla \F{1},A \rra \equiv \Tr (\O A \F{1}^T).
\end{equation*}
This definition satisfies the following connection with the contraction defined
in Eq.~\eqref{eq:contractionSimple}: indeed, for all $A\in\dTO(\mesh),$
\begin{equation*}
\quad \int_M \ii_A \F{1} \; dV = \; \lla \F{1},A \rra.
\end{equation*}
Note that the volume form $\O$ is needed to integrate the
piecewise-constant $0$-form $\ii_A \F{1}$ over the entire domain as explained in
Section~\ref{sec:0forms}. Finally, since the matrix $\O A$ is antisymmetric,
the symmetric component of $\F{1}$ does not play any role in the pairing.

Therefore, we will assume hereafter that a \emph{discrete one-form $\F{1}$ is defined by an antisymmetric matrix}: $\F{1}\in \frakSO(N)$.

\smallskip

\noindent
{\bf Remark I.} When viewed as acting on vector fields in the NHC space $\S$, our representation of discrete $1$-forms coincides with the use of $1$-cochains on the dual of $\mesh$~\cite{DeKaTo2005}: the value $\F{1}_{ij}$ (resp., $\F{1}_{ji}$) can be understood as the integral of a continuous $1$-form $\Ff{1}$ on the oriented dual edge going from cell $C_i$ to cell $C_j$ (resp., from $C_j$ to $C_i$). However, our use of antisymmetric matrices extends this cochain interpretation. This will become particularly useful when $1$-forms need to be paired with vector fields that have the form of the commutator $[A,B]$ of two vector fields $A$ and $B$ both in $\S$ as in Eq.~\eqref{eq:pairingWithCommutator}.
\smallskip

\noindent
{\bf Remark II.}
Notice finally that we can also define the notion of contraction of the volume form $\Omega$ by a discrete vector field $A$ using $\ii_A \Omega = 2 \Omega A $. The resulting matrix can be thought of as a discrete two form encoding the \emph{flux} of $A$ over each mesh face as derived in Section~\ref{sec:AasFlux}. In the notation convention of~\cite{Hirani2003}, this would be called a ``primal'' $2$-form, while the $2$-forms we will work with in this paper are ``dual'' $2$-forms. We won't discuss these primal $2$-forms further in this paper (as the construction of a consistent discrete calculus of forms and tensors is a subject on its own), but it is clear that they naturally pair with dual $1$-forms $\F{1}$ (forming a discrete wedge product between primal $2$- and dual $1$-forms), numerically resulting in the same value as the discrete pairing $\lla \F{1},A \rra$.

\subsection{Discrete Two-forms.}
We extend our definition of one-forms to two-forms in a similar fashion: discrete
$2$-forms will be encoded as $3$-tensors $\F{3}_{ijk}$ that are completely antisymmetric, i.e., antisymmetric with respect to any pair of indices.
\smallskip

\noindent
\textbf{Discrete Contraction.}
Contraction of a $2$-form $\F{2}$ by a vector field $A$ is defined as:
$$\lmat \ii_A \F{2} \rmat_{ij}=\sum_k \left( \F{2}_{ikj} A_{ik} - \F{2}_{jki} A_{jk}\right).$$
Notice again here that the resulting discrete $1$-form is indeed an antisymmetric
matrix (by construction), and that if $A \in \S$, many of the terms in the sum
vanish.
\smallskip

\noindent
\textbf{Discrete Total Pairing.}
The total pairing of a discrete $2$-form $\F{2}$ by two discrete vector fields $A$
and $B$, the discrete equivalent of $\int_M \Ff{2}(a,b) \; dV$, will be defined as:
\begin{equation}
\label{eq:pairing2Forms}
\lla \F{2}, A,B \rra \equiv 2 \sum_{i,j,k}\O_i \F{2}_{ijk}A_{ij}B_{ik}.
\end{equation}
This definition satisfies the expected property linking contraction and pairing: for
all $B\in\S,$
\[\lla \ii_A \F{2}, B \rra = \lla \F{2}, A,B \rra.\]
Indeed, using our previous definitions, we have:
{\allowdisplaybreaks
\begin{align*}
\lla \ii_A\F{2},B\rra &= \Tr (\O B (\ii_A\F{2})^T) =\sum_{i,j} \O_i B_{ij} (\ii_A\F{2})_{ij}\\
& = \sum_{i, j,k} \O_i \left( \F{2}_{ikj} A_{ik} - \F{2}_{jki} A_{jk} \right) B_{ij}\\
\left. \substack{\text{\small using }\F{2}_{ijk}=-\F{2}_{ikj}\\ \text{\small and }\O_iB_{ij} = - \O_jB_{ji}}\right\} &= \sum_i \left( \sum_{j,k} \left( - \O_i \F{2}_{ijk} A_{ik} B_{ij} + \O_j \F{2}_{jki} A_{jk} B_{ji}\right)\right)\\
&= -2\sum_i \left( \sum_{j,k}  \O_i \F{2}_{ijk} A_{ik} B_{ij} \right) = -\lla \F{2}, B,A \rra\\
& = \lla \F{2}, A,B \rra.
\end{align*}}

\subsection{Other Operators on Discrete Forms.}
A few more operators acting on $0$-, $1$-, or $2$-forms will be valuable to our
discretization of incompressible fluids.

\medskip

\noindent
\textbf{Discrete Exterior Derivative.} We can easily define a discrete version $\dd$
of the exterior derivative. For a discrete 0-form $\F{0},$ the one-form $\dd \F{0}$
is defined as
\[(\dd \F{0})_{ij}=\F{0}_j-\F{0}_i.\]
Similarly, if $\F{1}$ is a discrete one-form then we can define:
\[(\dd \F{1})_{ijk}=\F{1}_{ij}+\F{1}_{jk}+\F{1}_{ki}.\]
More generally, we define our operator $\dd$ as acting on a $k$-form $\F{k}$ through:
\[\lmat \dd \F{k} \rmat_{i_1 i_2 \hdots i_{k+1}} = \sum_{j\in [1..k+1]} (-1)^{j+1} \; \F{k}_{i_1 \hdots \widehat{ i_j } \hdots i_{k+1}}\]
where $\;\widehat{\cdot}\;$ indicates the omission of a term. This expression respects the antisymmetry of our discrete form representation.

\medskip

\noindent
{\bf Remark.}
 Notice here again that when the circumcenters of cells $C_{i_1}, C_{i_2}, \hdots,
C_{i_{k+1}}$ form a $k$-simplex on the dual of mesh $\mesh$, our definition of
$\dd$ simply enforces Stokes' theorem and thus coincides with the discrete
exterior derivative widely used in the literature~\cite{DeKaTo2005}. Our
discrete exterior derivative extends this simple geometric property to arbitrary
$(k+1)$-tuples of cells, while trivially enforcing that $\dd\circ \dd =0$ on the discrete level as well.

\medskip

\noindent
\textbf{Discrete Lie Derivative.} Now that we have defined contraction and
derivatives on discrete one-forms we can define the Lie derivative using Cartan's
``magic'' formula in the continuous setting $\mathbf{L}_v=\ii_v\dd+\dd\ii_v.$
\begin{Def}
Let $A$ be a discrete vector field satisfying the NHC and $\F{1}$ be a discrete one-form. Then the discrete Lie derivative
of $\F{1}$ along $A$ is defined as \[\mathbf{L}_A \F{1}=\ii_A\dd {\F{1}}+\dd\ii_A {\F{1}}.\]
\end{Def}

\begin{lemma}\label{LieDer}
For a vector field represented through an $\O$-antisymmetric and null-row
$A$, and a discrete closed one-form represented as a null-row and
antisymmetric $\F{1}$:
\begin{equation}
\label{eq:LieDerDef}
\mathbf{L}_A \F{1}= \, [A,\F{1}\O]\O^{-1}\, = \, A\F{1} - (A \F{1})^T.
\end{equation}
\end{lemma}

\begin{proof}
As $A$ is null-row, we have
\[\sum_k{\F{1}}_{ij}A_{ik}={\F{1}}_{ij}\sum_kA_{ik} =0.\]
Therefore,
\begin{align*}
\lmat\ii_A\dd {\F{1}}\rmat_{ij}&=\sum_k(\lmat\dd
{\F{1}}\rmat_{ikj}A_{ik}-\lmat\dd
{\F{1}}\rmat_{jki}A_{jk})\\
&=\sum_k({\F{1}}_{ik}+{\F{1}}_{kj}+{\F{1}}_{ji})A_{ik}-\sum_k({\F{1}}_{jk}+{\F{1}}_{ki}+{\F{1}}_{ij})A_{jk}\\
&=\lmat{A\F{1}}\rmat_{ij}+\lmat{\F{1}}A^T\rmat_{ii}-\lmat
{A\F{1}}\rmat_{ji}-\lmat{\F{1}}A^T\rmat_{jj}.
\end{align*}
Now, since $A^T=-\O A\O^{-1}$ and $\F{1}^T=-\F{1}$, we can write:
\[\lmat A {\F{1}}\rmat_{ji}=\lmat (A{\F{1}})^T\rmat_{ij}=\lmat\F{1} \O A\O^{-1}\rmat_{ij}\]
and, therefore,
\[\lmat\ii_A\dd {\F{1}}\rmat_{ij}=\lmat [A,{\F{1}}\O]\O^{-1}\rmat_{ij}+\lmat {\F{1}}A^T\rmat_{ii}- \lmat {\F{1}}A^T\rmat_{jj}.\]
Also, one has
\[\ii_A\F{1}=\diag(\F{1}A^T),\]
therefore,
\[\lmat\dd\ii_A\F{1}\rmat_{ij}=\lmat\F{1}A^T\rmat_{jj}-\lmat\F{1}A^T\rmat_{ii},\]
which implies the result.
\end{proof}

Note that the resulting formula corresponds to an antisymmetrization of $A$ applied to $\F{1}$---leading, up to the volume form $\O$, to the commutator of $A$ and $\F{1}$.

\subsection{Discrete $L^2$-inner Product and Discrete Flat Operator.}\label{Flattening}
The Lagrangian for incompressible, inviscid fluid dynamics is the squared $L^2$-norm of the velocity field. Hence, we wish to define a discrete $L^2$-inner product between two discrete vector fields. Since we require spatial sparsity (NHC condition) of the velocity field $A$, and Lin constraints for its variation $\delta A=\dot B+[A,B],$ we are only concerned with vector fields in $\S\cup[\S,\S]$.

Recall that the continuous flat of a vector field $v$ is a 1-form $v^\flat$ such that
\[ \la v^\flat,w\ra= (v,w),\quad\text{for every vector field }w, \]
where $(v,w)$ is the $L^2$-inner product of vector fields. Since the discrete total pairing is essentially a Frobenius inner product, discretizing the $L^2$ inner product for vector fields is equivalent to discretizing the flat operator $\flat:A\mapsto A^{\flat}$ such that the pairing of matrices $\lla A^{\flat},B\rra$ approximates the inner product of vector fields integrated on $M$:
\[
\linner A,B\rinner = \lla A^{\flat},B\rra=\Tr(\O B (A^{\flat})^T)\underset{h\to 0}{\rightarrow}\int_M
(v,w)\;dV,\,\,\text{if }A\tendsto{} v \text{ and } B \tendsto{} w.
\]

Looking ahead, we will only use the $L^2$ inner product of the type $\linner A\!+\!\delta A, A\!+\!\delta A \rinner$ when taking variations of the Lagrangian. Therefore, we need only to define $L^2$ inner products of the form $\linner A,B \rinner$, $\linner A,[B,C] \rinner$ and $\linner [B,C],A \rinner$, for any $A, B, C\in \S$ (equivalently, $\lla A^\flat, B\rra$, $\lla A^\flat, [B,C]\rra$, and $\lla [B,C]^\flat, A\rra$). As our discrete $L^2$ inner product will be symmetric, we only need to focus on inner products of the form $\linner A, \cdot \rinner$ (resp., $\lla A^\flat, \cdot \rra$) for $A\in\S$. Note that this discrete $L^2$ inner product can \emph{not} be trivial: indeed, for any matrices $A,\,B,\,C\in\S$, we have $\Tr(A[B,C])=0$ because $\S\cap[\S,\S]=\{{\mathbf 0}\}$; but we could choose $A$, $B$ and $C$ that approximate vector fields $v$, $u$ and $w$ such that $\int_M(v,[u,w])\ne 0$. As we now introduce, we define our discrete symmetric $L^2$ inner product in a matter that satisfies a discrete version of the continuous identity $\int_M(v, [u,w])=-\int_M d v^\flat(u,w)$, which holds for divergence-free vector fields.

\begin{Def}
Consider a family of meshes $\mesh_h$ of size $h$. An operator
\[\flat_h:\S\to\dTO(\mesh_h)\]
is called a discrete flat operator if the following two conditions are satisfied:
\begin{multline}\label{flat1}
\lla A_h^{\flat_h},B_h\rra\to\int_M (v(x),u(x))dx,\quad\text{when }h\to 0,\\
\text{for every }A_h,\,B_h\in\S,\,A_h\to \mathbf{L}_v,B_h\to \mathbf{L}_u
\end{multline}
\begin{multline}\label{flat2}
\lla A_h^{\flat_h},[B_h,C_h]\rra\to \int_M (v(x),[u,w](x))dx,\quad\text{when }h\to
0,\\
\text{for every }A_h,\,B_h,\,C_h\in\S,\,A_h\to \mathbf{L}_v,B_h\to \mathbf{L}_u,\,C_h\to \mathbf{L}_w.
\end{multline}
\end{Def}
\noindent Note that in this definition, $\lla A^{\flat_h},X\rra$ approximates the continuous inner product both when $X\in\S$ \emph{and} when $X\in[\S,\S]$.

\medskip

The next lemma introduces a necessary and sufficient condition to guarantee the validity of a
discrete flat operator. This particular condition will be very useful when we study
the dynamics of discrete fluids, as it involves the vorticity $\omega=\nabla \times u$ of a vector field:
\begin{lemma}\label{flatteninglemma}
A family of operators $\flat_h$ satisfies condition~(\ref{flat2}) if and only if for
every $A_h,\,B_h,\,C_h\in\S$ approximating vector fields $v,\,u,\,w\in\SVect(M)$
respectively we have
\[
\lla \dd A_h^{\flat_h},B_h,C_h\rra\to\int_M\omega(u,w) dV,\quad\text{where }\omega=\dd v^\flat.
\]
\end{lemma}
\begin{proof}
First, let's show that for any $u,\,v,\,w\in\SVect(M)$
\[\int_M(v,[u,w])dx=\int_M -\dd v^\flat(u,w)dx.\]
Indeed, since
\[\int_M(v,[u,w])dx=\int_M\ii_{[u,w]}v^\flat\]
and (see~\cite{MaRa1999})
\[\ii_{[u,w]}v^\flat=\mathbf{L}_u\ii_wv^\flat-\ii_w\mathbf{L}_uv^\flat,\]
we have
\[\int_M(v,[u,w])dx=\int_M\mathbf{L}_u\ii_wv^\flat-\ii_w\mathbf{L}_uv^\flat\overset{\nabla \cdot u = 0}{=}-\int_M\ii_w\mathbf{L}_uv^\flat.\]
But, by Cartan's formula $\mathbf{L}_uv^\flat=\ii_u\dd v^\flat+\dd\ii_uv^\flat$. Therefore,
\[\int_M\ii_w\mathbf{L}_uv^\flat=\int_M\ii_w\ii_u\dd v^\flat=\int_M\dd v^\flat(u,w),\]
where we used the fact that $w\in\SVect(M)$ and therefore $\int_M
\ii_w\dd\ii_uv^\flat=0$.\\

Now, let's show that $\lla\dd A^\flat,B,C\rra=-\lla A^\flat,[B,C]\rra.$
Using properties of the trace operator we have
\[\lla A^{\flat},[B,C]\rra=\Tr(\O [B,C](A^\flat)^T)=-\Tr(A^\flat\O[B,C])=-\Tr([A^\flat\O,B]C).\]
By Lemma~\ref{LieDer}, $[A^\flat\O,B]\O^{-1}=-\mathbf{L}_BA^\flat$. Thus,
\[-\Tr([A^\flat\O,B]C)=\Tr((\mathbf{L}_B A^\flat)\O C)=-\Tr(\O C (\ii_B\dd A^\flat)^T)=-\lla\dd A^\flat,B,C\rra,\]
where we used that $\Tr(\O C(\dd \ii_B A^\flat)^T)=0$ because $C$ is divergence-free.
\end{proof}

\medskip

\noindent
\textbf{Discrete Vorticity in the Sense of DEC.}
As our derivation relies on having a predefined notion of discrete vorticity, we first provide a definition used in~\cite{Hirani2003,ElToKaScDe2005} (we will refer to it as the DEC vorticity, as it was derived from a Discrete Exterior Calculus~\cite{DeKaTo2005}):
\begin{equation}
\omega_{\text{DEC}}(e) = \sum_{\substack{(i,j) \\ e \subset (C_i\cap C_j)}} 2 \O_i\frac{|e_{ij}|}{|S_{ij}|}A_{ij} s_{ij},
\end{equation}
where $s_{ij}=1$ if the cells $C_i$ and $C_j$ are positively oriented around $e$ and $s_{ij}=-1$ otherwise. Notice this represents the integral of the vector field $A$ along dual edges $e_{ij}$ around the edge $e$: by Stokes' theorem, $\omega_{\text{DEC}}(e)$ is thus \emph{the vorticity of $A$ integrated over the dual Voronoi face to $e$} (see Figure~\ref{AFlatPic}, left). More importantly, it has been established that this approximation does converge (as long the mesh does not get degenerate) to the notion of vorticity in the limit of refinement~\cite{Bossavit1998}.

\begin{figure}[h!]
\centering
\includegraphics[width=4in]{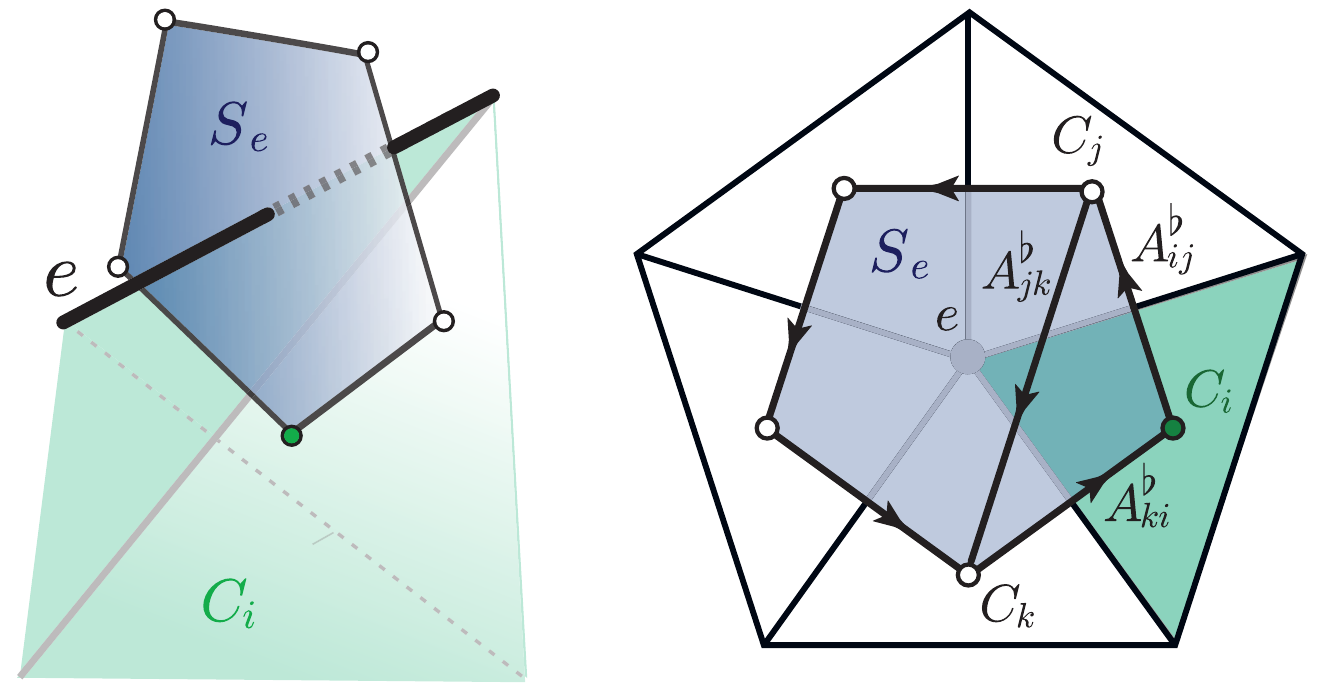}
\caption{\footnotesize Flat Operator: schematic representation of $A_{ij}^\flat$ as a part of the cell $S_e$ dual to edge $e$ in 3D (left) and a view of the dual cell seen straight along the edge (right). This last figure can also be seen as the 2D schematic version of the flat operator, where $e$ is now a vertex and $S_e$ is its associated dual Voronoi face.} \label{AFlatPic}
\end{figure}

\medskip

\noindent
\textbf{A Flat Operator on a 3-dimensional Mesh.}
From the previous lemma, we can derive a construction of a flat operator on a 3-dimensional
simplicial mesh. Given a matrix $A$, we need to find a matrix $A^\flat$ which satisfies the properties
\[ \lla A^\flat,A\rra = \lla A, A^\flat\rra\to\int_M \|v\|^2 \, dV,\]
and
\[\lla \dd A^\flat,B,C\rra\to\int_M\omega(u,w) \, dV.\]
To satisfy the first property we simply define the values of $A^\flat$ for \emph{immediate} neighbors as:
\begin{equation} \label{eq:flatAdjacent}
A^\flat_{ij}=A_{ij}\O_i\frac{2 |e_{ij}|}{|S_{ij}|}\quad\text{for }j\in N(i).
\end{equation}
Notice that it corresponds to the flux $2 \O_i A_{ij}$ of the velocity field, further multiplied by the diagonal Hodge star of $2$-forms for the face $S_{ij}$ (see, e.g.,~\cite{Bossavit1998}) to make $A^\flat$ a $1$-form on the dual edge between $C_i$ and $C_j$.

\smallskip

Enforcing the second property of the flat operator is more difficult; our construction will use the fact that in the limit, one must have
\[\int_M\omega(u,w) \, dV=\int_M *\omega\wedge u^\flat\wedge w^\flat.\]

\noindent Let's assume that the values of $A^\flat$ for adjacent cells are defined by Eq.~\eqref{eq:flatAdjacent}, and that the values of $A^\flat$ for \emph{non-adjacent} pairs
of cells $C_j$ and $C_k$ are \emph{defined} by:
\begin{equation}
\label{eq:flatNonAdjacent}
\lmat \dd A^\flat\rmat_{ijk}=A^\flat_{ij}+A^\flat_{jk}+A^\flat_{ki}
=K_{ijk}\;\omega_{\text{DEC}}(e_{ijk}),
\end{equation}
where $C_i$ is adjacent to both $C_j$ and $C_k$ (see Figure~\ref{AFlatPic} (right) for a schematic depiction), $e_{ijk}$ is the primal edge common to the cells $C_i$, $C_j$, $C_k$, and $K_{ijk}$ is a coefficient of proportionality whose exact expression will be provided later on. In other words, we assume that the flat operator allows us to evaluate vorticity not only on dual (Voronoi) faces as in the DEC sense, but on any triplet of cells $C_i, C_j, C_k$ as depicted in Figure~\ref{AFlatPic} (right); this will give us values of vorticity on \emph{subparts} of Voronoi faces as well.

Then the pairing $\lla \dd A^\flat,B,C\rra$ can be written (see Def.~\ref{eq:pairing2Forms}) as
\[\lla \dd A^\flat,B,C\rra=2 \sum_{i,\,j,\,k}\O_iK_{ijk}\;\omega_{\text{DEC}}(e_{ijk})\; B_{ij}C_{ik},\]
or, if one uses the flat of both vector fields $B$ and $C$,
\begin{equation} \label{dAflat}
\lla \dd A^\flat,B,C\rra=\half\sum_{i,\,j,\,k}\O_i\widetilde{K}_{ijk}\;\omega_{\text{DEC}}(e_{ijk})\; B^\flat_{ij}C^\flat_{ik},
\end{equation}
where
\[B^\flat_{ij}=\O_i B_{ij}\frac{2|e_{ij}|}{|S_{ij}|},\quad C^\flat_{ik}=\O_i C_{ik}\frac{2|e_{ik}|}{|S_{ik}|}, \;\;
\text{ and } \widetilde{K}_{ijk}=K_{ijk}\frac{1}{\O_i^2}\frac{|S_{ij}|}{|e_{ij}|}\frac{|S_{ik}|}{|e_{ik}|}.\]

Now, suppose we have a discrete version of the wedge
product (e.g.,~\cite{Hirani2003,SeSeSeAd2000}) between two dual one-forms, written with given weights $W_{ijk}$  as
\[
(B^\flat\wedge C^\flat)_{S_{e_{ijk}}}=\sum_{\substack{i,j,k\\ e_{ijk}=C_i\cap C_j\cap C_k}} W_{ijk} B^\flat_{ij} C^\flat_{ik},
\]
where $S_{e_{ijk}}$ is the two-dimensional face dual to the primal edge $e_{ijk}$ and the sum is taken as before over all consecutive cells $i$, $j$ and $k$ which have $e_{ijk}$ as a common edge. If we further define
\[\widetilde{K}_{ijk}=2W_{ijk}\frac{|e|}{\O_i |S_{e_{ijk}}|}\]
(where, as usual, $|e_{ijk}|$ denotes the length of the edge $e_{ijk}$ and $|S_{e_{ijk}}|$ is the area of the dual face $S_{e_{ijk}}$), we can reexpress equation \eqref{dAflat}
by summing over all edges $e_{ijk}$ to find a simple wedge-product-based version of the
total pairing of  the vorticity with two vector fields:
\[ \lla
\dd A^\flat,B,C\rra=\sum_e\omega_{\text{DEC}}(e)\frac{|e|}{|S_e|}(B^\flat\wedge
C^\flat)_{S_e}\approx \int_M(*\o)\wedge u^\flat\wedge w^\flat.
\]
Thus, we can derive the flat operator $\flat$ once a set of coefficients $W_{ijk}$ is known: given
a vector field $A\in\S$, $A^\flat$ for adjacent cells is defined using equation \eqref{eq:flatAdjacent}, while the rest of its non-zero values are defined such that
\[
A^\flat_{ij}+A^\flat_{jk}+A^\flat_{ki} = K_{ijk}\;\omega_{\text{DEC}}(e_{ijk}),
\]
where:
\[
K_{ijk}=2W_{ijk}\O_i\frac{|e_{ij}|\;|e_{ik}|}{|S_{ij}|\;|S_{ik}|}\frac{|e_{ijk}|}{|S_{e_{ijk}}|} \quad \text{for } e=C_i\cap C_j\cap C_k.
\]

A concrete expression of $W_{ijk}$ can be used by extending the definition of the primal-primal wedge product given in~\cite{Hirani2003} (Definition 7.1.1) to the dual in a straightforward fashion to make it exact for constant volume 2-forms through:
\[W_{ijk}=s_{ijk}\frac{|S_{e_{ijk}}\cap C_i|}{|\Delta_{ijk}|},\]
where $\Delta_{ijk}$ is a triangle with vertices at the (circum)centers of the cells $C_i$, $C_j$ and $C_k$, and $s_{ijk}=1$ if the triplet of cells $C_i,C_j,C_k$ is positively oriented around $e$ and $s_{ijk}=-1$ otherwise.

To simplify the expression for $K_{ijk}$ we use the equality
$|\Delta_{ijk}|=\frac12 |e_{ij}|\;|e_{ik}|\sin\a_{ijk}$, where $\a_{ijk}$ is the angle
between dual edges, yielding:
\[
K_{ijk}=4s_{ijk}\O_i|e|\frac{1}{\sin\a_{ijk}}\frac{1}{|S_{ij}|\;|S_{ik}|}\frac{|S_{e_{ijk}}\cap C_i|}{|S_{e_{ijk}}|}.
\]
Now, applying the generalized law of sines for the volume of a tetrahedron yields
\[\O_i=\frac{2}{3|e_{ijk}|}|S_{ij}| \; |S_{ik}| \; \sin\a_{ijk}\]
and thus
\[K_{ijk}=\frac{8}{3}s_{ijk}\frac{|S_{e_{ijk}}\cap C_i|}{|S_{e_{ijk}}|}.\]
This formula was used in the implementation of our method as described in~\cite{MuCrPaToDe2009} (note that the wedge product was rewritten as a function of the flux $F_{ij} = 2\O_iA_{ij}$).

\medskip

\noindent
\textbf{Flat Operator on Regular Grids in 2D.}
Our construction of the flat operator is particularly simple for regular (Cartesian)
grids as we now review for completeness.
\begin{lemma}
\label{lem:FlatRegular}
For a domain represented with a \emph{Cartesian grid} of size $h$, let $A$ be an
antisymmetric doubly-null matrix satisfying the NHC. The operator
$\flat:A\mapsto A^{\flat}$ defined as
\begin{equation*}
A_{ij}^\flat = 2 h^2A_{ij},\quad\text{for $i\in N(j)$,}
\end{equation*}
\begin{equation*}
A_{ij}^\flat = h^2\sum_{k\in N(i)\cap
N(j)} (A_{ik}+A_{kj}),\quad\text{for $i\ne N(j)$}
\end{equation*}
is a discrete flat operator.
\end{lemma}

Note that while $A$ satisfies the NHC, $A^\flat$ has non-zero elements
for neighboring cells and for cells that share a common neighbor (i.e., two cells
away). Now, let $i,\,k,\,j,\,l$ be four cells on a regular mesh sharing a common
node $x$, oriented counter-clockwise (see figure \ref{AFlatRegularpic}). Then it is easy to see that for $A^\flat$
defined above we have
$$\dd A^\flat_{ijk}=h^2(A_{ik}+A_{kj}+A_{jl}+A_{li})=\frac{\omega_{\text{DEC}}(x)}{2},$$
where $\omega_{\text{DEC}}(x)$ is the discrete vorticity in the sense of
Discrete Exterior Calculus integrated over the dual cell of node
$x$. Since $\omega_d$ converges to vorticity the condition of Lemma~\ref{flatteninglemma}
is satisfied, just as in the simplicial mesh case.

\begin{figure}[h!]
\centering
\includegraphics[width=1.5in]{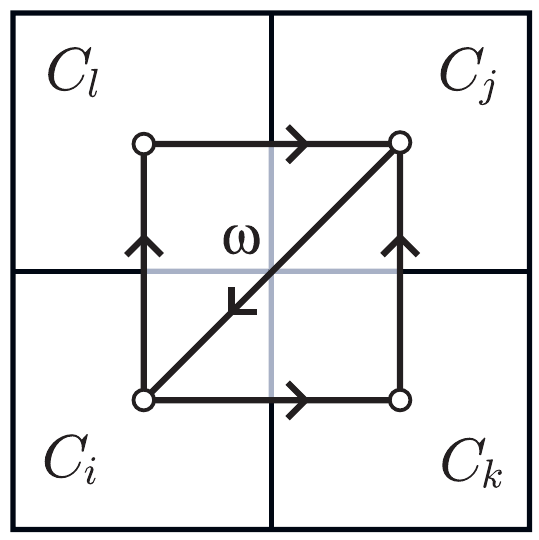}\vspace*{-2mm}
\caption{\footnotesize Flat Operator on a Regular Grid in 2D: our definition of the flat operator is particularly simple
when the spatial discretization is a regular mesh.\vspace*{-4mm}} \label{AFlatRegularpic}
\end{figure}

\section{Dynamics on the Group of $\O$-orthogonal Stochastic Matrices}
We now focus on defining a Lagrangian on the tangent bundle of the group $\TO(\mesh)$
of $\O$-orthogonal, signed stochastic matrices and studying the corresponding
variational principle with non-holonomic constraints. We will first assume a
discrete-space/continuous-time setup before presenting a fully discrete version.

\subsection{Variational Principle and Symmetry.}
We wish to study dynamics on the Lie group $\TO(\mesh)$ of $\O$-orthogonal, signed stochastic matrices representing volume-preserving diffeomorphisms on a mesh $\mesh$. While the group's Lie algebra $\dTO(\mesh)$ consists of null-row $\O$-antisymmetric matrices, we restrict the Eulerian velocity $A=- \dot qq^{-1}$ to lie in the NHC space $\S$, i.e., with fluid transfer happening only between adjacent cells (see Section~\ref{sec:NHC}).

We first establish a discrete Lagrangian $\L_h(q,\dot q)$ on $T\TO(\mesh)$ with the property that
$\L_h\longto{h\to 0}\frac12 \int \|v\|^2$ for $A \tendsto{h\to 0} v$ by defining:
\begin{equation*}
\L_h(A)=\frac12 \llangle A^\flat,A\rrangle\equiv\frac12\Tr (\Omega A (A^\flat)^T).
\end{equation*}

When $A$ satisfies the NHC, it was shown in Section~\ref{Flattening} that $\lla A^\flat, A\rra\to \int (v,v)$; thus the discrete Lagrangian is a proper approximation to the $L^2$-norm of the velocity field in this case. Note also that it is trivially right invariant as in the continuous case, since one can compose $q$ by a discrete diffeomorphism $\eta$ without changing the Eulerian velocity $A=-(\dot{q\eta})(q\eta)^{-1}=-\dot q q^{-1}$. Our discrete setup thus respects \emph{particle relabelling symmetry}.

\subsection{Computing Variations.}
\label{DiscreteLinConstraints}
To compute the variation of $A(t)$, we assume that
$q$ depends on a parameter $s$, we denote $q'=\frac{dq}{ds}$ and $\dot
q=\frac{dq}{dt}$, and we differentiate the Eulerian velocity:
\[ \frac{d}{ds}A(s,t)=-\dot q'q^{-1}+\dot qq^{-1}q'q^{-1}. \]
If we denote by $B$ the vector field satisfying $B=-q'q^{-1},$ we directly
get the well-known \emph{Lin constraints}:
\begin{equation} \frac{d}{ds}A(s,t)=\dot B+[A,B], \label{eq:discreteLin}\end{equation}
where $[A,B]=AB-BA$ is the commutator of matrices.

Now remember that the dynamics of systems with non-holonomic constraints
can be derived from the \emph{Lagrange-d'Alembert principle}:
\begin{equation}
\delta\int_0^1 \L_h(A)\; dt=0,\quad \delta q\in S_q,\quad A\in S,\quad \delta q(0) = \delta q(1) = 0
\label{eq:Discreted'AlembertPrinciple}
\end{equation}

Since $\delta q\in S_q$, the vector field $B$ must be in $\S$, i.e., $B_{ij}=0$
except for neighboring cells $C_i$ and $C_j$. We can then compute $\delta \L_h$:
\[
\delta \L_h(A)=\frac12\left(\lla \delta A^\flat,A\rra+\lla A^\flat,
\delta A\rra\right)=\lla A^\flat, \delta A\rra.
\]
As we restrict $A$ to lie in the NHC subspace $\S$, the Lin constraints in
Eq.~\eqref{eq:discreteLin} imply
\[ \delta \L_h(A)=\lla A^\flat,\dot B+[A,B]\rra. \]
Recall that if $A$ approximates $\mathbf{L}_v$ and $B$ approximates $\mathbf{L}_\xi$, then by definition of $\flat$,
\[\lla A^\flat,\dot B\rra\to \int_M (v,\dot\xi) \,dV\]
and
\[\lla A^\flat,[A,B]\rra\to \int_M (v,[v,\xi]) \,dV.\]
Thus,
\[\delta \L_h(A)\to \int_M(v,\dot\xi+[v,\xi]) \,dV=\delta l(v),\]
so the discrete Lagrangian (resp., its variation) is an approximation of the
continuous Lagrangian (resp., its variation).\\

Since $A^\flat$ is antisymmetric and $\Tr(A[B,C])=\Tr([A,B]C)$ for any matrices $A,\,B,\,C$,  we get
\[\delta \L_h(A)=\Tr\bigl(-A^\flat\O\dot B - [A^\flat\O,A]B\bigr).\]
After integration by parts and because variations are zero at each extremity of the time interval $[0,1]$, the discrete Euler-Lagrange equations of Eq.~\eqref{eq:Discreted'AlembertPrinciple} are:
\begin{equation}
\label{eq:DELraw}
\forall B\in\S, \quad \Tr\bigl((\dot A^\flat\O+[A, A^\flat\O])B\bigr)=0.
\end{equation}
\smallskip

To express the resulting equations in a more intuitive fashion, we introduce the following lemma:
\begin{lemma}
\label{lemma:pressure}
If matrix $Z \in \mathcal{M}^N$ is antisymmetric, with $\Tr(ZY^T)=0$ for every $Y\in\S$ then there exists a discrete pressure field, i.e., a vector $P=(p_1,\ldots,p_N)$ such that
\[ Z_{ij}=p_j-p_i,\quad\text{where }j\in N(i). \]
\end{lemma}
\begin{proof}
Since $Y\in\S$, the inner product of matrices $\Tr(Z Y^T)$ does not depend on
$Z_{ij}$ when $i$ and $j$ are not direct neighbors. We can thus assume that
$Z\in\N1$. The space $\S$ has codimension $N-1$ in the space $\N1$. Indeed, it
is defined by a system of $N-1$ independent equations:
\[ \sum_{j\in N(i)} Y_{ij}=0,\quad 1\le i\le N-1, \]
the last equation for $i=N$ being automatically enforced by the others.

Moreover, the space of discrete gradients (i.e., matrices $M \in \N1$ such as $M_{ij}=p_j-p_i$) is orthogonal to the space of null-row antisymmetric matrices w.r.t. the Frobenius inner product $M\cdot Y=\Tr(MY^T)$ and has dimension $N-1$. Therefore, the orthogonal complement to $\S$ in $\N1$ coincides with the space of discrete gradients.
\end{proof}

We directly deduce our main theorem:
\begin{theorem}
Consider the discrete-space/continuous time Lagrangian on $T\TO(\mesh)$:
\[ \L_h(A)=\frac12\lla A^\flat,A\rra, \]
where $A=-\dot qq^{-1} \in \S$ is a sparse, null-row, and $\O$-antisymmetric matrix, $q\in \TO(\mesh)$ is a signed stochastic $\O$-orthogonal matrix, and $A^\flat$ is the discrete flat operator defined in Section~\ref{Flattening} applied to $A$. Then the Lagrange-d'Alembert principle
\[ \delta\int_0^1 \L_h(A)\; dt=0,\quad \delta q\in S_q,\quad A\in S,\quad \delta q(0) = \delta q(1) = 0 \]
implies
\begin{equation}\label{EL}
\lmat\dot A^\flat+\mathbf{L}_AA^\flat+\dd p\rmat_{ij}=0,\quad\text{for }j\in N(i),
\end{equation}
or, equivalently,
\begin{equation}\label{EL2}
\dot A_{ij}+\frac{|S_{ij}|}{2\O_i\O_j
|e_{ij}|}[A,A^\flat\O]_{ij}+\frac{|S_{ij}|}{2\O_i
|e_{ij}|}(p_j-p_i)=0,\quad\text{for }j\in N(i)
\end{equation}
where $p$ is a discrete pressure field to enforce $A \in \S$.
\end{theorem}
\begin{proof}
Apply Lemma~\ref{lemma:pressure} (for $Z =\dot A^\flat+[A, A^\flat\O]\O^{-1}$ and $Y=\Omega B$)
to Eq.~\eqref{eq:DELraw} and substitute the definition of discrete Lie derivative given
in Eq.~\eqref{eq:LieDerDef}.
\end{proof}

The resulting discrete Euler-Lagrange (DEL) equations we obtained represent a \emph{weak form} of
the continuous Euler equations expressed as:
\[ \dot v^\flat+\mathbf{L}_v v^\flat +\dd p=0. \]
Furthermore, these equations of motion can be reexpressed in various ways, mimicking
different forms of Euler equations: for instance, the discrete equations of motion
written as
\[
(\dot A^\flat+i_A \dd A^\flat +\dd \tilde p)_{ij}=0
\]
for all $ j\in N(i)$, which are equivalent to
\[
\dot v + v \times \omega +\nabla \tilde p =0
\]
(where $\tilde p$ is the dynamic pressure), while taking the exterior derivative
of these same equations leads to
\[
(\dot{(\dd A^\flat)} + \mathbf{L}_A (\dd A^\flat))_{ij} =0 \quad \forall j\in N(i),
\]
a discrete version of
\[
\dot \omega + \mathbf{L}_v \,\omega =0.
\]

\subsection{Discrete Kelvin's Theorem.}
\label{sec:Kelvin} %
This section presents a discrete version of Kelvin's theorem that the discrete Euler-Lagrange equations fulfill. Through replacing the continuous notions of a curve and its advection by discrete Eulerian counterparts, the proof of this discrete Kelvin's theorem will be essentially the same as in the continuous case, which we will describe first for completeness.

\medskip

\noindent
\textbf{Kelvin's Theorem: The Continuous Case.} %
Kelvin's theorem states that the circulation along a closed curve stays constant as the curve is \emph{advected} with the flow. Let $\gamma_t$ be a closed curve and $C_{\g_t}v_t$ be the circulation of $v_t$ along $\gamma_t$, i.e.:
\[
C_{\g_t}v_t=\oint_{\g_t}v_t\cdot ds.
\]
Consider a divergence-free vector field $\g^\e_0$ representing a ``narrow current" of width $\e$ flowing along $\g_0$, with unit flux when integrated over transversal sections of the curve. This \emph{current} can be thought of as an $\e-$spreading (akin to a convolution by a smoothed Dirac function) of the tangent vector field to the immediate surroundings of the curve $\gamma_t$, forming a smoothed notion of a curve. Let $\g^\e_t$ be the field $\g^\e_0$ advected by the flow $v_t$, i.e., it satisfies:
\begin{equation} \dot\g_t^\e+\mathbf{L}_{v_t}\g^\e_t=0. \label{eq:continuousAdvection}\end{equation}
Note that this equation encodes the notion of advection of a curve seen from a current point of view, hence without the need for a parameterization of the curve; see~\cite{Arnold1966}. Then, as $\e\to 0$,
\[
\la v^\flat_t,\g^\e_t\ra\to C_{\g_t}v_t,
\]
so the pairing $\la v^\flat_t,\g^\e_t\ra$ can be considered as an approximate circulation converging to the real circulation as $\e\to 0$. We can compute its derivative:
\[
\frac{d}{dt}\la v^\flat_t,\g^\e_t\ra=\la\dot v^\flat_t,\g^\e_t\ra+\la
v^\flat_t,\dot\g^\e_t\ra=-\la \mathbf{L}_{v_t}v^\flat_t,\g^\e_t\ra-\la v^\flat_t,\mathbf{L}_{v_t}\g^\e_t\ra=0.
\]
And since this pairing represents the circulation along the $\e-$smoothed curve for any $\epsilon$, the circulation itself stays constant.

\medskip

\noindent {\bf Remark.} A current is formally the dual of a $1$-form (in the sense of vector space duality), i.e., it is a linear map that takes a $1$-form to $\R$. When the space is equipped with a metric, one can think of a current as a vector field as described above. While a metric-independent treatment is possible as well, we will stick to the vector field point of view for simplicity in this paper.
\medskip

Achieving the goal of finding a discrete Kelvin's theorem first requires a definition of \emph{discrete curves} and their advection, for which we will borrow the concept of
\emph{one-chains} used in algebraic topology and demonstrate that curves and
vector fields satisfying the non-holonomic constraints share the same
representation; that is, \emph{the discretization of a curve $\gamma(s)$ will be thought
of as a discretization of the narrow current $\g^\e$}. Since we already have
established a discrete analog to the Lie derivative (based on the commutator of matrices), we
will be able to define how to \emph{advect a discrete curve along a discrete
vector field}. We will find that, just like Kelvin's circulation theorem in the
continuous case, for any discrete curve $\gamma_t$ advected by a discrete
vector field $A(t)$ satisfying the discrete Euler equations, the circulation of $A(t)$
along $\gamma_t$ remains constant.

\medskip

\noindent
\textbf{Discrete Curves.}
A discrete curve in our Eulerian setup can be nicely defined using the concept of one-chains. Let's recall that \emph{dual one-chains} are linear combinations of dual edges (linking two adjacent cells), converging to one-manifolds as the mesh gets finer (see~\cite{Munkres1984} for a thorough exposition of chains and simplicial homology, and~\cite{Bossavit1998} for applications in electromagnetism). In our context, in order to consider curves as ``currents'' (i.e., localized vector fields) as in the continuous description above, we will be using a linear combination of primal fluxes instead, exploiting the well-known isomorphism (from the Poincar{\'e} duality theorem) between dual one-chains and primal two-forms in 3D (i.e., between dual one-chains and primal $(n-1)$-cochains in dimension $n$, see~\cite{Munkres1984}). In other words, an $\O$-antisymmetric matrix will be used to describe a discrete curve as it was used to describe a two-form. We start by defining a \emph{simple} curve:

\begin{Def}
A {\bfi simple discrete curve} is a discrete path from cell $C_{i_1}$ to cell
$C_{i_2}, \ldots$, to cell $C_{i_n}$ with $C_{i_k}$ adjacent to $C_{i_{k+1}}$ and such that $(i_k, i_{k+1}) \ne (i_j, i_{j+1})$ for $k \ne j$.  It is represented by an $\O$-antisymmetric
matrix $\Gamma$ whose entries $\Gamma_{ij}$ satisfy
\[
\O_{i_k}\Gamma_{i_ki_{k+1}}=-\Gamma_{i_{k+1}i_k}\O_{i_{k+1}}=\frac12,
\]
and
\[
\Gamma_{ij}=0,\quad\text{for $(i,j)\ne(i_k,i_{k+1})$} \, \forall k.
\]
\end{Def}

The matrix $\Gamma$ representing a simple discrete curve $\gamma(s)$ that exactly
follows dual edges can be considered as a discrete current induced by the
tangent field $d\gamma(s)/ds$. Moreover, one can extend the notion of discrete
curves to encompass arbitrary dual one-chains. In our work, we will be focusing
on \emph{closed} discrete curves described as discrete \emph{divergence-free}
currents:
\begin{Def}
A {\bfi discrete closed curve} is a simple discrete curve that closes (i.e., a
discrete path from cell $C_{i_1}$ to cell $C_{i_2}, \ldots$, to cell $C_{i_n}$, and back to cell $C_{i_1}$).
It is represented by a null-row $\O$-antisymmetric matrix $\Gamma$ such that
$\Gamma_{ij}=0$ when $(i,j)\ne(i_k,i_{(k+1)\!\!\mod\!n})$ for some $k$.
\end{Def}
Since our discrete representation of a one-manifold coincides with our
definition of discrete Eulerian velocities in the NHC, we will no
longer distinguish between discrete curves and discrete velocities.

\medskip

\noindent
\textbf{Discrete Circulation.} Due to the duality between discrete curves and discrete fluxes of a vector field, the circulation of a vector field along a curve is trivially computed using the same pairing of matrices we used earlier:
\begin{Def}
The circulation $C_{\Gamma} A$ of a discrete vector field $A$ along a discrete curve $\Gamma$
is defined as
\[ C_\Gamma A\equiv\lla A^\flat, \Gamma\rra. \]
\end{Def}

We finally need to define a discrete notion of advection, which should be an approximation to $\mathbf{L}_v\gamma^\e$. We use a matrix $A\in\S$ to discretize $v$ and a matrix $\Gamma\in\S$ to discretize $\gamma^\e$, so their commutator $[A,\Gamma]$ is a discretization of $\mathbf{L}_v\gamma^\e$. However, $[A,\Gamma]\notin\S$. So instead, we can only consider the elements of $[A,\Gamma]$ that satisfy the constraints to define our \emph{weak} notion of curve advection:

\begin{Def}
Let $\Gamma_t\in\S$ be a family of discrete curves evolving in time and $A_t$ is a
(time-dependent) discrete vector field. We say that $\Gamma_t\in\S$ is advected by
$A_t$ if $\Gamma_t$ satisfies the advection equation
\begin{equation} \label{advection} \lla X^\flat,\dot \Gamma+[A,\Gamma]\rra=0,\quad\text{for any }X\in\S.\end{equation}
\end{Def}
Note that this definition defines a projection of the commutator $[A,\Gamma]$ onto the subspace $\S$ of non-holonomic constraints, and Fig.~\ref{fig:Projection} depicts this projection for the case of a regular grid.

\begin{figure}[h!]
\centering
\includegraphics{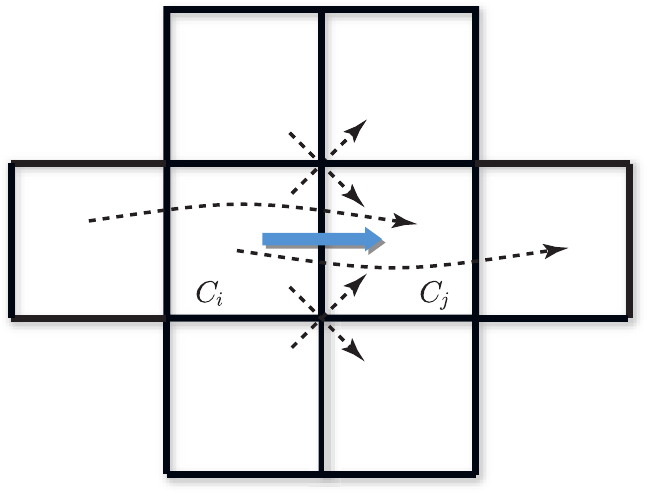}\vspace*{-4mm}
\caption{\footnotesize Projection on Regular Grids: our projection of $[A,B]$ onto
the subspace of non-holonomic constraints accumulates on the common
boundary of $C_i$ and $C_j$ all the two-cell-away transfers going
through this boundary. In this figure, the transfers in dotted lines are summed up
(with a weight of $\half$ for the diagonal ones) and assigned to the blue one-away
transfer.}\label{fig:Projection}
\end{figure}

Now, let's prove that if $\dot \Gamma$ satisfies Eq.~(\ref{advection}), it is a
discrete (weak) approximation of Eq.~\eqref{eq:continuousAdvection}. Indeed, if $X\rightsquigarrow w$, $A\rightsquigarrow v$, $\Gamma\rightsquigarrow\gamma$, then, by definition of the discrete operator
$\flat$,
\[
\lla X^\flat,\dot\Gamma\rra \to \int_M (w,\dot \g)
\]
and
\[
\lla X^\flat,[A,\Gamma]\rra \to \int_M (w,[v,\g]).
\]
Thus, if Eq.~(\ref{advection}) is satisfied, $\dot\g$ has to satisfy
\[\int_M(w,\dot\g+[v,\g])=0\]
for every $w\in\SVect(M)$. Since $\g\in\SVect(M)$, this last equation is a weak form of
$\dot\g=-[v,\g]=-\mathbf{L}_v\g$.

\medskip

\noindent
\textbf{Discrete Kelvin's Theorem.} We are now ready to give a discrete analog of
Kelvin's circulation theorem satisfied by our discrete Euler equations.
\begin{theorem}
If $A_t$ satisfies the DEL equations~\eqref{EL} and $\Gamma_0$ is an arbitrary
discrete curve, then the circulation of $A$ along $\Gamma_t$ stays constant:
\[ C_{\Gamma_t} A_t=C_{\Gamma_0} A_0, \]
where $\Gamma_t$ is the curve $\Gamma_0$ advected by $A_t$.
\end{theorem}
\begin{proof}
The time derivative of the circulation $C_{\Gamma_t} A_t$ is expressed as:
\[
\frac{d}{dt}C_{\Gamma_t}A_t=\frac{d}{dt}\lla A_t^\flat,\Gamma_t\rra=\lla\dot A_t^\flat,\Gamma_t\rra+\lla A_t^\flat,\dot \Gamma_t\rra.
\]
Since $A^\flat_t$ satisfies the DEL equations $\lmat \dot
A_t^\flat+[A_t^\flat\O, A_t]\O^{-1}+\dd\bar p_t\rmat_{ij}=0$ for $i$ and $j$
representing two neighboring cells' indices, and as $\Tr((\dd\bar p_t)\Gamma_t)=0$, we
have
\begin{align*}
\lla \dot A_t^\flat,\Gamma_t\rra=-\lla[A_t^\flat\O,A_t]\O^{-1},\Gamma_t\rra &=-\Tr([A_t^\flat\O,A_t]\Gamma_t)\\
& =-\Tr(A_t^\flat\O[A_t,\Gamma_t])=-\lla A_t^\flat,[A_t,\Gamma_t]\rra.
\end{align*}
But since $\Gamma_t$ is advected by $A$, we get
\[ \frac{d}{dt}C_{\Gamma_t}A_t=\lla\dot A_t^\flat,\Gamma_t\rra+\lla A_t^\flat,\dot \Gamma_t\rra=\lla A_t^\flat,\dot \Gamma_t+[A_t,\Gamma_t]\rra=0.\vspace*{-6mm}\]
\end{proof}

\noindent
{\bf Remark.} In the continuous case, the Kelvin's circulation theorem can be
derived from Noether's theorem using right-invariance of the metric on $\SDiff$
(particle relabelling symmetry). In the discrete case, the Lagrangian is also
right invariant, but the presence of the non-holonomic constraints prevents us
from using Noether's theorem directly: in a system with non-holonomic constraints
a momentum is no longer expected to be conserved in general. However, we can still
use the symmetry to obtain the momentum equation, i.e. the rate of change of the
momentum in time. Doing so for our discrete fluid model we also get our discrete
circulation theorem.

\section{Fluid evolution in discrete time}
\label{sec:discretetime}

In this section, we revisit our discrete version of the variational principle discussed
above by making time discrete instead of continuous. We assume that the fully
discrete fluid motion is given as a discrete path $q_0, q_1,\ldots,q_K$ in the space of
$\O$-orthogonal signed stochastic matrices, where the motion has been
sampled at regular time $t_k=k \tau$ for $k \in \{0,1,\ldots, K\}$, $\tau$ being referred
to as the time step size.

\smallskip

\subsection{Discrete Velocity.}
Given a pair $q_k, q_{k+1}$ of consecutive configurations in time, we can compute a
discrete time analog of Eulerian velocity using, e.g., one of the following
classical formulas:
\begin{align*}
\bullet\quad q_{k+1} & = q_k -\tau A_k \; q_k,                &\text{(explicit Euler)}\\
\bullet\quad q_{k+1}& = q_k  -\tau A_k \; q_{k+1},            &\text{(implicit Euler)}\\
\bullet\quad q_{k+1}& = q_k  -\tau A_k\; (q_k+q_{k+1})/2.     &\text{(midpoint rule)}
\end{align*}
Note that the midpoint rule preserves the Lie group structure of the configuration
space. Note also that many other discretizations could be used, but we will
restrict our explanations to the first two cases
as they suffice to illustrate how
our continuous time procedure can be adapted to the purely discrete case.

\begin{figure}[h!]
\centering
\includegraphics[width=3in]{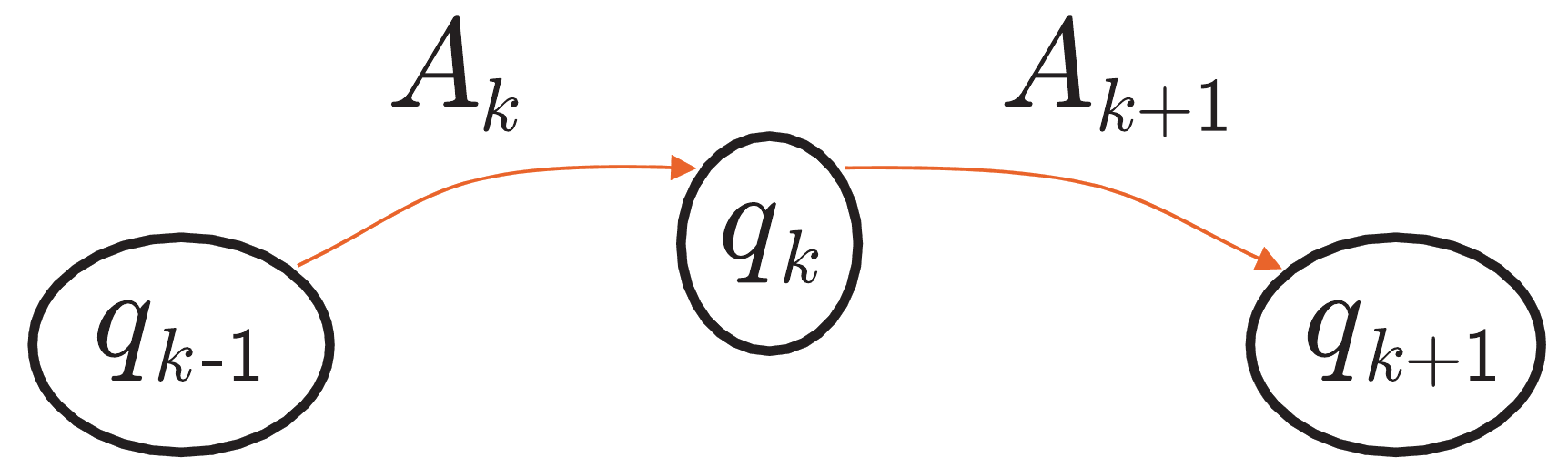}
\caption{\footnotesize Three consecutive configurations $q_{k-1}, q_k, q_{k+1}$ of a fluid in time, with Eulerian velocities $A_k$ and $A_{k+1}$ in between.} \label{states}
\end{figure}

\subsection{Discrete Lagrangian and Action.}
We define the discrete-space/discrete-time Lagrangian $\L_d(q_k,q_{k+1})$ as
\[ \L_d(q_k,q_{k+1})=\L_h(A_k). \]
The discrete action $\A_d$ along a discrete path is then simply the sum of all
pairwise discrete Lagrangians:
\[ \A_d(q_0,\ldots,q_K)=\sum_{k=0}^{K-1} \L_d(q_k,q_{k+1}). \]
We can now use the Lagrange-d'Alembert principle that states that $\delta
\A_d=0$ for all variations of the $q_k$ (for $k=1,\dots,K-1$, with $q_0$ and
$q_K$ being fixed) in $S_q$ while $A_k$ is restricted to $\S$.

\medskip

\noindent
\textbf{Variations.}
The variations of $A_k$ can be easily derived:
\begin{itemize}
\item {\it Explicit Euler.}
In this case, $A_k=-(q_{k+1}-q_k)/\tau\;q_k^{-1}$. The variation $\delta_k A_k$ and $\delta_{k+1}A_k$ with respect to $q_k$ and $q_{k+1}$ respectively become:
\[ \delta_k A_k=\frac{1}{\tau} \delta q_k q_k^{-1}+\frac{q_{k+1}-q_k}{\tau} q_k^{-1}\delta q_k q_k^{-1},\]
\[ \delta_{k+1} A_k=-\frac{1}{\tau} \delta q_{k+1}q_k^{-1}. \]
If we denote, similar to the continuous case, $B_k=-\delta q_k q_k^{-1}$, we
get:
\[ \delta_k A_k=-\frac{B_k}{\tau}+A_kB_k \]
and
\[ \delta_{k+1} A_k=\frac{B_{k+1}}{\tau}-B_{k+1}A_k.\]

\item {\it Implicit Euler.}
In this case $A_k=-(q_{k+1}-q_k)/\tau\;q_{k+1}^{-1}$. It yields:
\[ \delta_kA_k=\frac{1}{\tau} \delta q_kq_{k+1}^{-1} \]
and
\[ \delta_{k+1}A_k= -\frac{1}{\tau} \delta q_{k+1}q_{k+1}^{-1}+\frac{q_{k+1}-q_k}{\tau}q_{k+1}^{-1}\delta q_{k+1}q_{k+1}^{-1}. \]
Similarly to the previous case we now obtain:
\[ \delta_kA_k=-\frac{B_k}{\tau}-B_kA_k,\]
and
\[ \delta_{k+1}A_k=\frac{B_{k+1}}{\tau}+A_kB_{k+1}. \]

\end{itemize}

\subsection{Discrete Euler-Lagrange Equations.}
\label{sec:DEL}
Equating the variations of the action
$\A_d$ with respect to $\delta q_k$ to zero for $k\in[1,K-1]$ yields:
\begin{equation}\label{DEL}
\delta_k\lla A^\flat_{k-1},A_{k-1}\rra+\delta_k\lla
A^\flat_k,A_k\rra=0.
\end{equation}
Thus we obtain:
\[\Tr\big[A_{k-1}^\flat\O(\delta_kA_{k-1})+A_k^\flat\O(\delta_kA_k)\big]=0.\]

Now, let's solve it for $A_k$ in the explicit Euler case. Substituting the expressions for
$\delta_kA_k$ and $\delta_k A_{k-1}$ yields:
\[ \Tr\bigl[A_{k-1}^\flat\O(B_k-\tau B_kA_{k-1})+A_k^\flat\O(-B_k+\tau A_kB_k)\bigr]=0.\]
Denoting $\dot A_k^\flat=(A_k^\flat-A_{k-1}^\flat)/\tau$ we can rewrite the last
equation as
\[\Tr[(\dot A_k^\flat\O+A_{k-1}A_{k-1}^\flat\O-A_k^\flat\O A_k)B_k]=0.\]

Therefore, we get the following discrete Euler-Lagrange equations in the
explicit Euler case:
\[\dot A_k^\flat-\left(\frac{A_k^\flat\O
A_k\O^{-1}-A_{k-1}A_{k-1}^\flat}{2}-\frac{(A_k^\flat\O
A_k\O^{-1}-A_{k-1}A_{k-1}^\flat)^T}{2}\right)+\dd p_k=0.\]
As $(A_k^\flat\O A_k\O^{-1})^T = A_k A_k^\flat$ and $(A_{k-1}A_{k-1}^\flat)^T=A_{k-1}^\flat\O A_{k-1} \O^{-1}$, this last expression is equivalent to
\begin{equation} \label{TimeDiscreteEL}
\dot A_k^\flat+\frac12([A_{k-1},A_{k-1}^\flat\O]\O^{-1}+[A_k,A_k^\flat\O]\O^{-1})+\dd p_k=0,
\end{equation}
corresponding to the discrete-time version of Eq.~\eqref{eq:DELraw}.

Using the implicit Euler formula for $A_k$ instead of the explicit Euler one leads to the exact same equation,
we thus omit the computations here.

\subsection{Update Rule for Regular Grids in 2D}
The discrete Euler equation we derived above turns out to be particularly simple when applied to a regular grid. Indeed,
let's consider a regular grid of size $h$, on a two-dimensional
domain and with continuous time for simplicity. Then the discrete Euler equation~\eqref{EL2} becomes
\[2 h^2 \dot A_{ij}+[A,A^\flat]_{ij}+(p_j-p_i)=0,\quad\text{for }j\in N(i).\]
Now let's fix $i$ and $j$ and expand $[A,A^\flat]_{ij}$. Since $A\in\S$ we have
\[ [A,A^\flat]_{ij}=\sum_{l\in N(i)}A_{il}A^\flat_{lj}-\sum_{k\in N(j)}A^\flat_{ik}A_{kj}.\]
From the definition of $A^\flat$ (see Lemma~\ref{lem:FlatRegular}) we get:
\[A^\flat_{ik}=-\frac12\o_{ik}s_{ijk}+2h^2(A_{ij}+A_{jk}),\quad\text{for }k\in N(j) \text{ and } k \notin N(i)\]
and
\[ A^\flat_{lj}=-\frac12\o_{lj}s_{jil}+2h^2(A_{ij}+A_{li}),\quad\text{for }l\in N(i) \text{ and } l \notin N(j),\]
where $\o_{i_1i_2}$ is the vorticity in the DEC sense computed at the common node of cells $i_1$ and $i_2$ if $i_1$ and $i_2$ have a common node (see Fig.~\ref{AFlatRegularpic}), and $0$ otherwise; also $s_{i_1i_2i_3}=1$ if the triplet of cells $i_1,\,i_2,\,i_3$ is oriented counter-clockwise and $s_{i_1i_2i_3}=-1$ otherwise.

Now the equations for the commutator $[A,A^\flat]$ become
\[
[A,A^\flat]_{ij}=\frac12\sum_{k\in N(j)}A_{kj}\o_{ik}s_{ikj}-\frac12\sum_{l\in
N(i)}A_{il}\o_{lj}s_{jil}+2h^2\sum_{k\in N(j)}A_{jk}^2-2h^2\sum_{l\in
N(i)}A_{il}^2.
\]
If $k\in N(l)$ then $\o_{ik}=\o_{lj}$, so only two $\o$'s are
present in the expression above. Let's denote them by $\o_-$ and
$\o_+$ as depicted in Fig.~\ref{UpdateRegular}, and write
\[
[A,A^\flat]_{ij}=-\o_-\frac{A_{jk_1}+A_{il_1}}{2}-\o_+\frac{A_{k_2j}+A_{l_2i}}{2}+Q_j-Q_i,
\]
where $Q_i=2h^2\sum_{l\in N(i)}A_{il}^2$.

\begin{figure}[h!]
\centering
\includegraphics[width=1.5in]{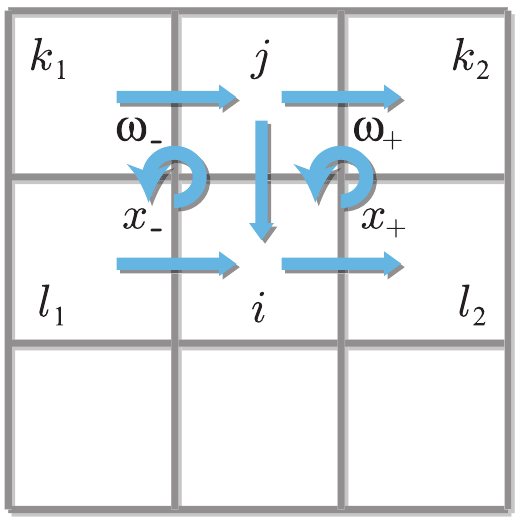}\vspace*{-2mm}
\caption{\footnotesize Notations used to rewrite the discrete Euler equation
on a regular mesh as a function of local velocities and vorticities.}
\label{UpdateRegular}
\end{figure}

As we know, $\o_-/h^2$ and $\o_+/h^2$ approximate the values $\o(x_-)$,
$\o(x_+)$ of vorticity at the corresponding nodes. Also, $A_{ij}\approx
-v_{ij}/2h$. Now suppose the pair of cells $C_i$ and $C_j$ is oriented along the
$y$ direction (see Fig.~\ref{UpdateRegular} again) and $v=(v_1,v_2)$. Let's
denote $A_{ij}\approx -v_2/2h$ and
\[
\begin{array}{rcl}
2hA_{ik_1} & = & -v_1^{--}\\
2hA_{jk_1} & = & -v_1^{-+}
\end{array}
\begin{array}{rcl}
2hA_{k_2i} & = & -v_1^{+-}\\
2hA_{k_2j} & = & -v_1^{++}
\end{array}.
\]
Now, the discrete discrete Euler equation implies
\[
\dot
v_2+\frac14(\o(x_-)(v_1^{--}+v_1^{-+})+\o(x_+)(v_1^{+-}+v_1^{++}))+P_j-P_i=0,
\]
where $P$ is some discrete function, playing the role of pressure. This
equation, together with the equations for every pair $i$ and $j$, is
a discrete version of the two-dimensional Euler equations written in
the form
\[
\begin{array}{rcl}
\dot v_1-\o v_2+P_x & = & 0,\\
\dot v_2+\o v_1+P_y & = & 0,
\end{array},\quad\div v=0,\quad \o=\frac{\partial v_2}{\partial
x}-\frac{\partial v_1}{\partial y}.
\]

The discretization of the Euler equations that we have obtained on the regular grid coincides with the Harlow-Welsh scheme (see~\cite{HaWe1965}), and Eq.~\eqref{TimeDiscreteEL} is a Crank-Nicolson (trapezoidal) time update. Therefore, our variational scheme can be seen as an extension of this approach to arbitrary grids, offering the added bonus of providing a geometric picture to these numerical update rules.

\begin{figure}[ht]
\centering
\includegraphics[width=\textwidth]{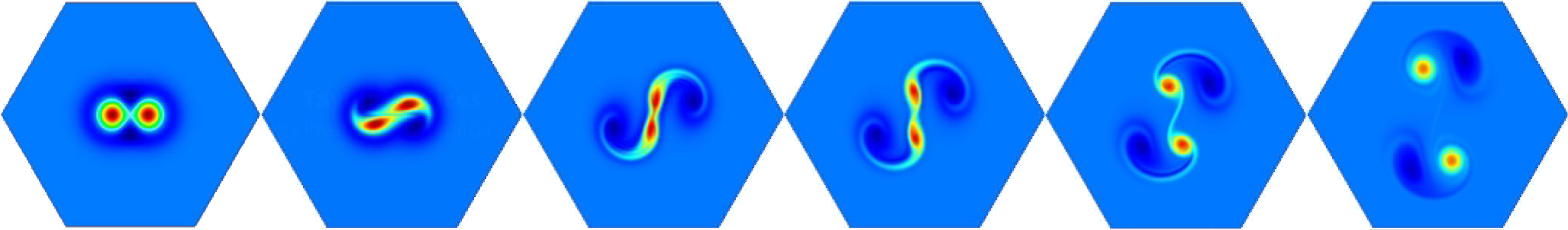}\vspace*{-2mm}
\caption{\footnotesize \textbf{Taylor Vortices Separating:} two like-signed Taylor vortices (with a finite
vorticity core) will merge when their distance of separation is smaller than some critical value. In this example,
the two vortices in a domain discretized with 55296 triangles were initialized at a distance slightly above this
critical value, leading to a separation.}\vspace*{-1mm}
\label{fig:results}
\end{figure}

\section{Conclusions and Discussions}

The discrete geometric derivation of Euler equations we presented above differs sharply from previous geometric approaches. First and foremost, our work derives the fluid mechanics equations from the least action principle, while many previous techniques are based on finite volume, finite difference, or finite element methods applied to Euler equations(see~\cite{HaWe1965,Perot2000,GrSa2000} and references therein). Second, our derivation does not presume or design a Lie derivative or a Poisson bracket in the manner of geometric approaches such as~\cite{Salmon2004}. Finally, we discretize the volume-preserving diffeomorphism group, offering a purely Eulerian alternative to the inverse map approach proposed in~\cite{CoHoHy2007,CoHo2009}. The resulting scheme does, however, have most of the numerical properties sought after, including energy conservation over long simulations~\cite{Perot2000}, time reversibility~\cite{DuOrWi2008}, and circulation preservation~\cite{ElToKaScDe2005}. We now go over some of the computational details and present a few results, before discussing possible extensions.

\subsection{Computational Details.}
Implementing the discrete update rules derived in Section~\ref{sec:DEL} is straightforward: from the sparse matrix $A_k \in \S$ describing the velocity field at time $t_k$, a new sparse matrix $A_{k+1} \in \S$ is computed by solving the discrete Euler-Lagrange equations through repeated Newton steps until convergence. Notice that the configuration state $q_k$ is not needed in the computations, making the numerical scheme capable of directly computing $A_{k+1}$ from $A_k$. Moreover the update rules can be rewritten entirely as a function of fluxes $F_{ij}=2 \Omega_i A_{ij}$, rendering the assembly of the advection operator simple. Finally, the Newton steps can also be made more efficient by approximating the Jacobian matrix involved in the solve by only its diagonal terms. Further details on the computational procedure can be found in~\cite{MuCrPaToDe2009}, including linearizations of the discrete Euler-Lagrange equations that tie our method with~\cite{SiAr1994}. It is also worth mentioning that our simulator applies to arbitrary topology, and that viscosity is easily added by incorporating a term proportional to the Laplacian of the velocity field~\cite{MuCrPaToDe2009}.

\begin{figure}[h!]
\centering
\includegraphics[width=.8\textwidth]{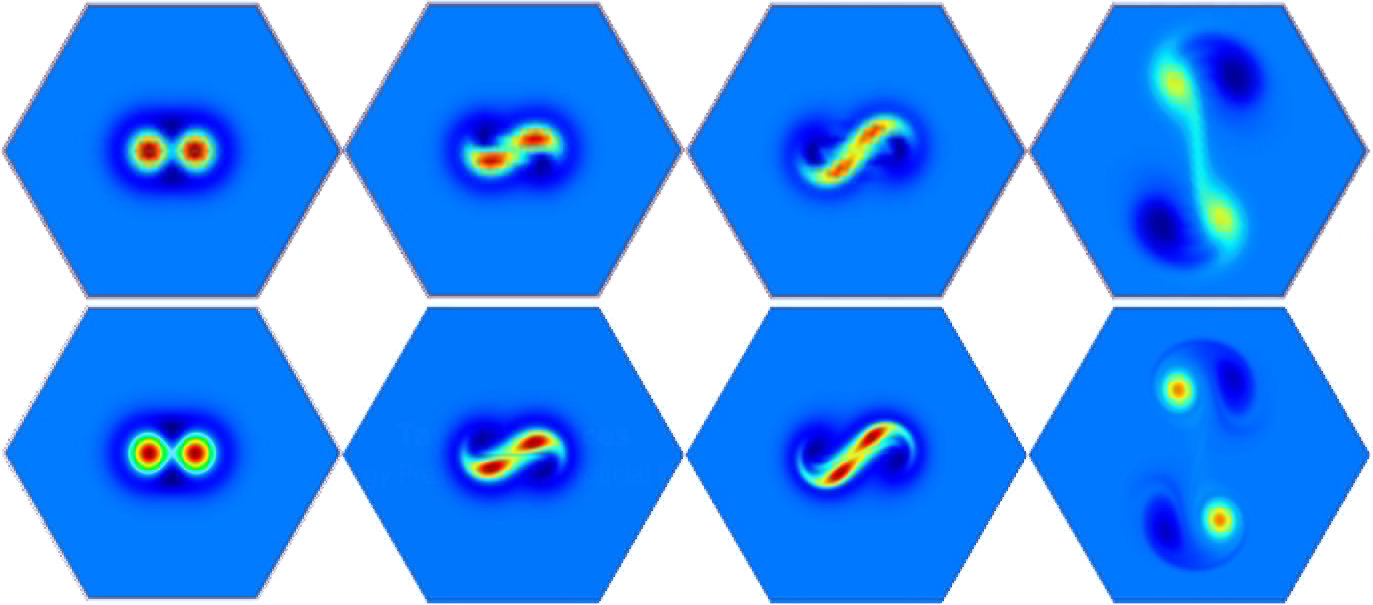}\vspace*{-1mm}
\caption{\footnotesize 2D fluid simulation of Taylor vortices separating: even at two very
different resolutions, our variational scheme leads to similar results (top:
4056 triangles; bottom: 55296 triangles; same continuous initial conditions discretized
on both grids).\vspace*{-2mm}}
\label{fig:CoarseVsFine}
\end{figure}

\subsection{Numerical Tests and Results.}
As expected from the time reversibility of the resulting discrete Euler-Lagrange equations, our fully Eulerian scheme demonstrates excellent energy behavior over long simulations, even for very low thresholds on the Newton solver. This numerical property was well known for the Harlow-Welsh discretization over regular grids when using a trapezoidal time integration scheme; our approach extends this scheme and its properties to arbitrary mesh discretization. Figure~\ref{fig:results} shows the results of this geometric integrator on a
\begin{wrapfigure}{r}{0.4\textwidth}
\hspace*{-2mm}\vspace*{1mm}\centering
\includegraphics[width=0.4\textwidth] {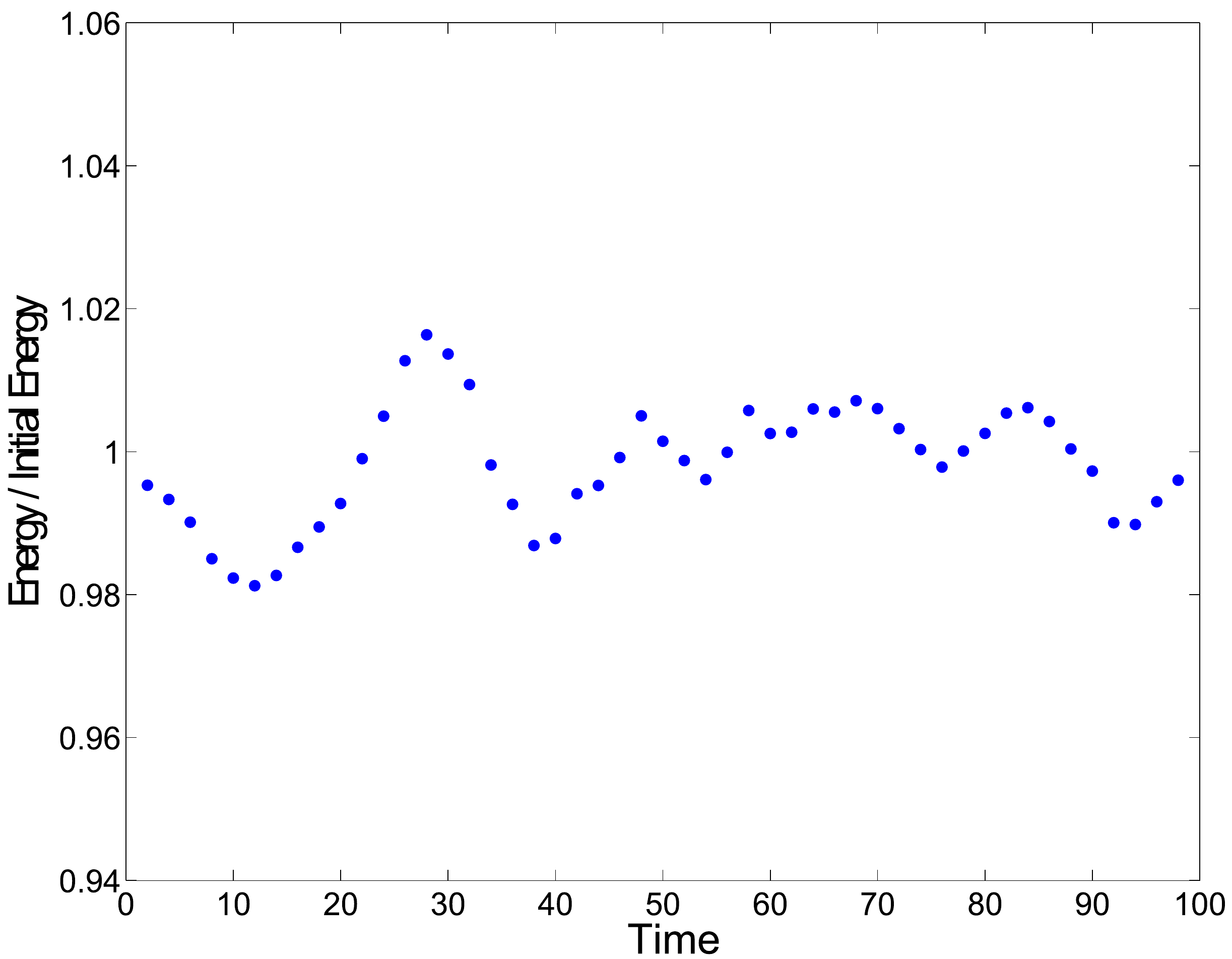} \vspace*{-6mm}\end{wrapfigure}
common test used in CFD, where a periodic 2D domain is initialized with two Taylor vortex distributions of same sign placed at a distance close to a critical bifurcation in the dynamics: as expected, the two vortices eventually separates, and our integrator keeps the energy close to the initial energy over extended simulation time (see inset). Figure~\ref{fig:CoarseVsFine} demonstrates the robustness of the integrator to grid size: the same dynamics of the vortices is still captured even on a number of triangles thirteen times smaller. Finally, Figure~\ref{fluid_teaser} shows frames of a simulation of a three-dimensional fluid on a tetrahedral mesh.

\subsection{Extensions.}
The results of this paper are rich in possible extensions. For instance, generalizing our approach to higher-order integrators is an obvious research direction. A midpoint approximation of the Eulerian velocity between $q_k$ and $q_{k+1}$ preserves the Lie group structure of the configuration space, but leads to additional cubic terms in the variation $\delta A$, thus requiring a flat operator valid for three-away cells as well. Finding a systematic approach to deriving such higher-order updates is the subject of future work.

We could also investigate alternative expressions for the discrete Lagrangian. One possibility is to notice that in the continuous case, the Lagrangian can be written as
\[
\L(q,\dot q) = \frac12 \sum_{i=1}^{n}(\mathbf{L}_v x_i,\mathbf{L}_v x_i)
\]
where $x_i$ represents the $i^\text{th}$ coordinates in $\R^n$. Its discrete equivalent in 2D could therefore be written as $\Tr(A^2 (X+Y))$ where the matrix $X\in\S$ (resp., $Y\in\S$0 is expressed as $X_{ij}=x_{C_i} \cdot x_{C_j}$ (resp., $Y_{ij}=y_{C_i} \cdot y_{C_j}$), with $x_{C_k}$ (resp., $y_{C_k}$) represents the $x$-coordinate (resp., $y$-coordinate) of the circumcenter of cell $C_k$. Taking variation would lead to $\Tr(A\, \delta A\, (X+Y)) = \Tr( ((X+Y)\dot{A} - [(X+Y)A, A])B)=0 \, \forall B \in \S$. We see that this alternate definition of the Lagrangian defines another flat operator (albeit, in a less geometric way).

Similarly, one may change the sparsity requirement of the NHC by defining the space $\S$ to be the sparsity induced by adjacency through \emph{vertices}. It would require a new Lie bracket which is not directly the Lie bracket for the matrices representing the vector fields, but the sparsity constraint would no longer be non-holonomic.

We wish to look at how the energy of our discrete simulator cascades at lower scales. More generally, understanding what this geometric picture of fluid flows brings compared to traditional Large Eddy Simulation or Reynolds-Averaged Navier-Stokes methods would be interesting, as our structure-preserving approach is also based on local averages (i.e., integrated values) of the velocity field.

We also wish to investigate the use of an ``upwind'' version of $A$, possessing only positive fluxes as often used in the discretization of hyperbolic partial differential equations~\cite{Leveque2002}. This would allow the reconstruction of non-negative matrices $q_k$, making them transition matrices of a Markov chain.

Finally, we note that the geometric understanding developed here should offer good foundations to tackle related problems, such as magnetohydrodynamics, variable density fluids, or Burgers' equations. Our initial results using an extension to systems with semi-direct product group structure show promise.


\end{document}